\documentclass{article}
\usepackage[margin=1.00in]{geometry}
\usepackage{xcolor}
\usepackage{graphicx} 
\usepackage{yhmath} 
\usepackage{stmaryrd}
\usepackage{xcolor}
\usepackage{subcaption}
\usepackage{mathtools}
\usepackage{relsize}
\usepackage{amsmath}
\usepackage{amsfonts}
\usepackage{amsthm}
\usepackage{stmaryrd}
\usepackage{cleveref}
\usepackage{tikz,tikz-cd}
\usetikzlibrary{arrows,shapes,cd, angles}
\usepackage[export]{adjustbox}
\usepackage{tcolorbox}
\usepackage{overpic}
\usepackage{ulem}
\usepackage{wrapfig}
\usepackage{caption}
\usepackage{authblk} 

\usepackage{subcaption}

\newtheorem{theorem}{Theorem}
\newtheorem{lemma}[theorem]{Lemma}

\DeclareMathOperator{\sinc}{sinc}

\newcommand{\dpart}[3]{\partial_{#1}^{#2}{#3}} 

\newcommand{\barp}{\overline{\psi}}
\newcommand{\bardual}{\overline{\varphi}}
\newcommand{\bx}{\boldsymbol{x}}
\newcommand{\BX}{\boldsymbol{X}}

\newcommand{\elmt}[1]{K_{#1}}
\newcommand{\edg}[1]{\mathbf{s}_{#1}}
\newcommand{\Pk}[1]{\mathbb{P}^{#1}}
\newcommand{\Qk}[1]{\mathbb{Q}^{#1}}
\newcommand{\R}{\mathbb{R}}
\newcommand{\D}{\mathcal{D}}
\newcommand{\sphe}{\mathbb{S}^2}
\newcommand{\Sang}{\mathcal{S}}
\newcommand{\ang}{\boldsymbol{\Omega}} 
\newcommand{\angj}[1]{\ang_{#1}}
\newcommand{\angx}[1]{v_{#1}}
\newcommand{\psij}[1]{\psi^{#1}}
\newcommand{\psijh}[1]{\psi^{#1}_h}

\newcommand{\dualj}[1]{\varphi^{#1}}
\newcommand{\sigs}{\sigma_s}
\newcommand{\siga}{\sigma_a}
\newcommand{\sigt}{\sigma_t}
\newcommand{\omj}{\omega_j}
\newcommand{\vard}[1]{(\bx,\ang_{#1},t)} 
\newcommand{\varx}[1]{(x,v_{#1},t)} 
 
\newcommand{\barf}[1]{\overline{#1}}
\newcommand{\barfh}[1]{\overline{#1}_h}
\newcommand{\Nang}{N_{\ang}} 
\newcommand{\NangO}{N_{\angx{}}} 
\newcommand{\Nx}{N_x} 

\newcommand{\dK}{\partial K}
\newcommand{\dx}{h}
\newcommand{\leg}[1]{L_{#1}}
\newcommand{\radp}[1]{R_{#1}^+}
\newcommand{\radm}[1]{R_{#1}^-}
\newcommand{\radpm}[1]{R_{#1}^\pm}
\newcommand{\I}{\mathcal{I}}
\newcommand{\map}{\zeta} 
\newcommand{\norma}[1]{\|{#1}\|}
\newcommand{\Lp}[3]{\norma{#1}_{L^{#3}({#2})}}

\newcommand{\Hs}[3]{\norma{#1}_{{#3},{#2}}}
\newcommand{\Oh}{\mathcal{O}}
\newcommand{\Th}{\mathcal{T}_h} 
\newcommand{\Eh}{\mathcal{E}_h} 
\newcommand{\Vh}[1]{V_h^{#1}} 
\newcommand{\normal}{\mathbf{n}}
\newcommand{\rgl}{\rangle}
\newcommand{\lgl}{\langle}
\newcommand{\Oavg}[1]{\lgl {#1} \rgl}  
\newcommand{\Kindx}{n}
\newcommand{\dgdeg}{k}
\newcommand{\dgorder}{k+1} 
\newcommand{\test}{\tau_h}
\newcommand{\flux}[2]{\widehat{\left({#1}\right)}_{#2}}

\newcommand{\dgjump}[2]{\llbracket{#1}\rrbracket_{#2}} 
\newcommand{\dgavg}[2]{\{\!\!\{{#1}\}\!\!\}_{#2}} 
\newcommand{\ehj}[1]{e_h^{#1}}
\newcommand{\ejn}[2]{e_{#2}^{#1}}
\newcommand{\bxn}[1]{\overline{x_{#1}}}
\newcommand{\apj}[3]{\alpha^{#1}_{{#2},{#3}}}
\newcommand{\Qpm}[2]{Q^{#1}_{{#2}}}

\newcommand{\modelEXd}[2]{\dpart{t}{}{\psij{#1}\vard{#1}}+\ang_{#2}\cdot\nabla_{\bx}\psij{#1}\vard{#1} = \sigs(\barf{\psi}(\bx,t) - \psij{#1}\vard{#1})-\siga\psij{#1}\vard{#1}} 
\newcommand{\modelEXx}[2]{v_{#2} \dpart{x}{}{\psij{#1}\varx{#1}} = \sigs(\barf{\psi}(\bx,t) - \psij{#1}\varx{#1})-\siga\psij{#1}\varx{#1}}

\newcommand{\modelSteadyOne}[2]{ \angx{#2} \dpart{x}{}{\psij{#1}\varx{#1}} + \sigt\psij{#1}\varx{#1}= \sigs\barf{\psi}(x,t) }
\newcommand{\quadf}[2]{\sum_{j=1}^{#1}\, \omj {#2}}

\newcommand\by{{\mathbf y}}

\newcommand\Kernel[3]{{K^{#1}_{#2}\left(#3\right)}}
\newcommand\knots[1]{{\text{\tiny{\textbf T}}_#1}}
\newcommand\bspline[2]{{\boldsymbol B_{#1}(#2)}}

\newcommand{\dt}{\Delta t}
\title{Superconvergence Extraction of Upwind Discontinuous Galerkin Method Solving the Radiative Transfer Equation}

\author[1]{Andr\'es Galindo-Olarte}
\author[2]{Zhichao Peng\footnote{Corresponding author: pengzhic@ust.hk}}
\author[3]{Jennifer K. Ryan}

\affil[1]{Oden Institute for Computational Engineering and Sciences, The University of Texas at Austin, Austin, TX, USA}
\affil[2]{Department of Mathematics, The Hong Kong University of Science and Technology, Clear Water Bay, Kowloon, Hong
Kong, China}
\affil[3]{Department of Mathematics, KTH Royal Institute of Technology, Stockholm, Sweden}

\begin{document}

\maketitle
\begin{abstract}
We theoretically analyze the superconvergence of the upwind discontinuous Galerkin (DG) method for both the steady-state and time-dependent radiative transfer equation (RTE), and apply the Smooth-Increasing Accuracy-Conserving (SIAC) filters to enhance the accuracy order. Direct application of SIAC filters on low-dimensional macroscopic moments, often the quantities of practical interest, can effectively improve the approximation accuracy with marginal computational overhead.

Using piecewise $k$-th order polynomials for the approximation and assuming constant cross sections, we prove $(2k+2)$-th order superconvergence for the steady-state problem at Radau points on each element and $(2k+\frac{1}{2})$-th order superconvergence for the global $L^2-$ and negative-order Sobolev norms for the time-dependent problem. 

Numerical experiments confirm the efficacy of the filtering, demonstrating post-filter convergence orders of $2k+2$ for steady-state and $2k+1$ for time-dependent problems. More significantly, the SIAC filter delivers substantial gains in computational efficiency. For a time-dependent problem, we observed an approximately $2.22\times$ accuracy improvement and a $19.94\times$ reduction in computational time. For the steady-state problems, the filter achieved a $4–9\times$ acceleration without any loss of accuracy.

\end{abstract}

\section{Introduction}

The radiative transfer equation (RTE) (\cite{pomraning2005equations,lewis1984computational,case1967linear}) describes particles, e.g. neutrons or photons,  propagating through and interacting with a background medium by means of a kinetic distribution that gives the density of particles with respect to the local phase space measure . It has a wide range of applications in nuclear engineering, medical imaging, astrophysics and remote sensing. Efficient numerical simulation tools for RTE are highly desired. In this article, we take advantage of the underlying superconvergent properties of the approximation to improve accuracy as well as computational performance.

In order to approximate solutions to RTE, the upwind discontinuous Galerkin (DG) method is used.  DG has been a popular and powerful deterministic solver for RTE since its first introduced by Reed and Hill in simulating neutron transport \cite{reed1973triangular}. A key strength of the high-order upwind DG method is its asymptotic preserving (AP) property, which ensures it correctly captures the diffusion limit of the RTE without resolving the small particle mean free path. This property was formally established through asymptotic analysis \cite{larsen1989asymptotic,adams2001discontinuous}, and later rigorously proved \cite{guermond2010asymptotic}. Uniform convergence with respect to the mean free path has been analyzed in \cite{sheng2021uniform,sheng2025numerical}. Besides upwind DG, DG methods have been actively developed for solving RTE. These developments include \emph{hp}-adaptive hybridized DG approach \cite{du2023fast},  approaches leveraging micro-macro decomposition and implicit-explicit (IMEX) time integrator   \cite{jang2014analysis,peng2020stability,peng2021asymptotic,peng2021stability,xiong2022high}, semi-Lagrangian approach \cite{cai2024asymptotic}, and  high-order low-order method (HOLO) \cite{feng2025decomposed}. 

Despite its success, the discontinuous Galerkin (DG) method for the radiative transfer equation (RTE) suffers from the curse of dimensionality, as the unknown distribution function is posed in a high-dimensional phase space. An attractive mitigation is post-processing techniques  that can significantly enhance resolution of numerical solutions. Exploiting pointwise or Fourier space superconvergence, post-processing techniques typically enhance accuracy of a finite element solution through a convolution that performs effective local averaging to damp dominating high frequency errors. The development of post-processing dates back to Bramble and Schatz \cite{bramble1977higher} in the context of continuous finite elements methods for elliptic problems, which exploits superconvergence in negative-order norms. Formally defined as weighted dual norms, negative-order norms can be seen as a bridge connecting information in the Fourier/signal space with the one in physical space.  For time-dependent problems, the DG method is proven to exhibit higher-order accuracy in negative-order norms— $2\dgorder$ for dispersion errors and $2\dgdeg+2$ for dissipative errors—compared to only $\dgdeg+1/2$ in the $L^2$-norm. The pioneering work of Cockburn, Luskin, Shu, and Süli \cite{cockburn2003enhanced} established the framework for superconvergence analysis in negative norms and post-processing for time-dependent linear hyperbolic equations. This relies on a properly designed kernel that respects moment \cite{mock1978}. This technique was later extended to nonlinear scalar conservation laws \cite{ji2013negative,ji2014superconvergent,meng2017discontinuous}, nonlinear symmetric systems \cite{meng2018divided}, and, more recently, to the collisionless Vlasov–Maxwell system \cite{galindo2023superconvergence}. Besides exploiting negative-order norm, super-convergence for time-dependent problems can also be proved through a direct Fourier analysis \cite{guo2013superconvergence}. Beyond signal space, superconvergence of DG methods for steady-state linear and nonlinear hyperbolic problems is analyzed through polynomial expansions of error equations \cite{adjerid2002posteriori,ADJERID20063331}.

In this paper, we exploit post-processing based on Smooth-Increasing Accuracy-Conserving (SIAC) filters \cite{steffan2008investigation,mirzaee2011smoothness,docampo2023magic} to enhance the accuracy of both steady state and time-dependent RTE, which is a collisional kinetic equation. In practice, the desired quantity of interest is often low-dimensional macroscopic moments (e.g., the density) rather than the high-dimensional particle distributions.  We theoretically establish the superconvergence results for both steady state and time-dependent RTE and numerically demonstrate that directly applying SIAC filters to these macroscopic moments significantly enhances their approximation accuracy.  Since the effectiveness of SIAC filter relies on the superconvergence of the DG scheme, we also establish superconvergence results for both steady state and time-dependent RTE. The main contributions of this paper are as follows.
\begin{enumerate}
    \item For the  steady-state problem, we prove $\dgdeg+2$ convergence order at Radau points within each element and $2\dgdeg+2$-th order at the outflow edge. 
    This result is established by extend the framework based on coefficient matching for a polynomial expansion of the error equation \cite{adjerid2002posteriori,ADJERID20063331} which proves $2\dgdeg+2$-th order convergence at outflow flux and $\dgdeg+2$-th order at interior Radau points for linear advection and nonlinear scalar hyperbolic problems. Instead of using a key argument inapplicable to RTE based on the theory for the Runge-Kutta method, we have a novel presentation of the analysis and show that the local polynomial is the same on each element and only weighted by the error coefficients.

    \item For the time-dependent problem with constant cross-section, we establish, for the first time, an $L^2-$ error estimate for the DG-DO method. By carefully designing a dual problem, we prove a $2\dgdeg+\tfrac{1}{2}$ superconvergence rate in a negative Sobolev norm and observe a $2\dgorder$ convergence rate numerically. This analysis, which addresses discretization in both space and angle for time-dependent problems, is new. Previous works focused either on steady-state problems \cite{han2010discrete,sheng2021spherical} or assumed continuity in physical space \cite{frank2016convergence,galindo2025numerical,chen2019multiscale}.

    \item Through a series of benchmark tests, we validate our theoretical results and showcase the SIAC filter's capacity to drastically improve computational efficiency. By enhancing the accuracy achievable on coarse meshes, the filter reduces the time required to reach a target error tolerance in comparison to mesh refinement. Specifically, we observe an approximately $3.26\times$ accuracy boost along with a $19.65\times$ computational time reduction for a time-dependent problem, and $4\times$ to $9\times$  speedup  \emph{without sacrificing accuracy} for a steady state problem.
\end{enumerate}

This paper is organized as follows. We introduce the model problem, notation, preliminaries and the discretization in Sec. \ref{sec:background}, our superconvergence results are presented in Sec. \ref{sec:superconvergence-filter}, sketch the key steps for our proofs in Sec. \ref{sec:proof_outline}.  The SIAC filter is introduced in Sec. \ref{sec:SIAC} and the numerical tests of our methods are presented in Sec. \ref{sec:num}.  Our conclusions are discussed in Sec. \ref{sec:con}. 







\section{Background\label{sec:background}}

We consider one-group, linear radiative transfer equation with isotropic scattering:
\begin{subequations}
\label{eq:rte_md}
\begin{align}
&\modelEXd{}{}, \label{eqn:hd_rt_eqn}\\
&\barp(\bx,t)=\Oavg{\psij{}\vard{}} = \frac{1}{4\pi}\int_{\mathbb{S}^2}\, \psij{}\vard{} \, d\ang,  \label{eqn:sphere_ave}
\end{align}
\end{subequations}
equipped with proper boundary conditions. Here, $\psij{}\vard{}$ is the particle distribution function (also known as intensity or angular flux) for spatial location $\bx\in\BX\subset\R^3$, angular direction on the unit sphere $\ang\in\mathbb{S}^2$ and time $t$. The macroscopic density (also known as scalar flux), as $\barp(\bx,t)=\Oavg{\psij{}\vard{}}=\frac{1}{4\pi}\int_{\mathbb{S}^2}\, \psij{}\vard{}\, d\ang$ is the average of the distribution function with respect to angular direction, and $\Oavg{\cdot}$ represents the averaging operator over angular space. Here, $\sigs(\bx)\geq 0$ is an isotropic scattering cross section, while $\siga(\bx)\geq0$ is the absorption cross section.  The total cross section is defined as $\sigt(\bx)=\sigs(\bx)+\siga(\bx)$. In this work we will assume the cross-sections $\sigt$ and $\siga$ are constant in physical space $\bx$.

In our analysis for the steady-state problem, we focus on the one-dimensional slab geometry. For a one-dimensional slab geometry, the particle distribution $\psij{}\vard{}$ only depends on a location $x\in\mathbb{R}$ and the cosine of angle between the angular direction and the $x$-axis $\ang \in [-1,1]$. With this assumption, the equation can be simplified as $v=\cos(\Omega)$:
\begin{subequations}
\label{eq:1d_slab}
\begin{align}
& \modelEXx{}{}, \\ 
& \barp(x,t)=\frac{1}{2}\int_{-1}^{1}\, \psij{}(x,v)\,dv.
\end{align}
\end{subequations}

On the other hand, following \cite{galindo2025numerical}, to help the error analysis in the time, since the cross-sections $\siga$ and $\sigt$ are constant in space, we can reduce the analysis to Equation \eqref{eqn:hd_rt_eqn} to a purely scattering system for the function  (abusing notation) 
\begin{equation*}
\text{``}\psi := e^{\siga t}\psi\text{''}.
\end{equation*}
$\psi$ satisfies
\begin{equation} 
\begin{split}
    &\partial_t\psij{}\vard{}+\ang\cdot\nabla_{\bx}\psij{}\vard{} =\sigma(\barf{\psij{}}(\bx,t)-\psij{}\vard{}),\\
    &\barf{\psij{}}(\bx,t)  =\frac{1}{4\pi}\int_{\Sang}\psij{}\vard{}\,d\ang
\end{split}
\label{eq:hd_rt_reduced}
\end{equation}
where $\sigma:=\sigs$. Henceforth, we restrict our attention to \eqref{eq:hd_rt_reduced}. The corresponding results for the case of nonzero absorption can then be recovered by reversing the transformation, which introduces exponential time decay when $\siga > 0$.

In order to define the numerical discretization and prove superconvergence, we first note a few preliminaries.

\subsection{Notation and preliminaries}

In this subsection, we discuss notation and estimates that will be used throughout this article.

\subsubsection{Notation}

\textbf{Spaces and Norms.} In this article, we denote the domain for angular space as $\Sang = \mathbb{S}^{d-1}$ and for Euclidean space as $\BX\subset \R^d.$ $\BX$ refers to a general domain -- either in Euclidean or angular space. Hence, when $\BX\subset\R^3,\, \Sang = \Sang.$

We start by defining the multi-index $\alpha$, where
$\alpha = (\alpha_1,\dots,\alpha_d),\quad  |\alpha| = \sum_{j=1}^{d}\, \alpha_j  $
and 
\[
D^{|\alpha|}f(\bx) = \dpart{x_1}{\alpha_1}{}\dpart{x_2}{\alpha_2}{}\cdots\dpart{x_d}{\alpha_d}{f(\bx)},
\]
where $f$ is some given function.  

The usual $L^p-$norm is defined as
\[
\Lp{f}{\BX}{p} = \left(\int_{\BX}\, |f|^p\, d\bx\right)^{1/p},
\]
where $\Lp{f}{\D}{\infty} = \text{ess sup}_\D\, |f|$.

The general Sobolov space for $k \in \mathbb{N}_+,\, 1\leq p \leq \infty$ is defined as
\begin{equation}
    W^{k,p}(\BX) :=\{v\in L^p(\BX):\, D^{|\alpha|}f \in L^p(\BX),\, |\alpha|\leq k\},
\end{equation}
with norm
\begin{equation}
    \norma{f}_{ W^{k,p}(\D)} = \sum_{|\alpha|\leq k}\, \Lp{f}{\D}{\Kindx}.
\end{equation}
We note that when $p=2,$ we denote $H^k = W^{k,2}$ we simplify notation as
\[
\Hs{f}{\BX}{k} = \norma{f}_{H^{k}(\BX)}
\]
and define the $L^{2}-$inner product as
\[
(f,g)_{\BX} = \int_{\BX}\, fg\, d\bx.
\]

Further, for purposes of error analysis contained in this article, we define a weighted negative-order Sobolev norm for $F=(f^1,\cdots,f^N)$ as
\begin{equation}
    \norma{F}_{-\ell,\BX}=\sup_{\Phi\in (C_0^{\infty}(\BX))^N} \left[\sum_{j=1}^N\omega_j(f^j,\phi^j)_{\BX}\right]/\left(\sum_{j=1}^N\omega_j\norma{\phi^j}_{\ell,\BX}^2\right)^{1/2}.
    \label{eqn:negative-order-norm}
\end{equation}

The weights $\omega_j$ correspond to the quadrature weight ones in the $S_N$ discretization for the phase-angle space; a more detailed discussion will be provided below.

\textbf{Dual problem.}  To obtain superconvergence for the time-dependent equation, we will need to bound the negative-order Sobolev norm of the divided differences of the error.  This requires defining the dual problem.
The dual problem for $j=1,\dots,\Nang$ is defined as
\begin{equation}
\begin{split}
&\partial_t\dualj{j}+\angj{j}\cdot\nabla_x\dualj{j}+\sigma(\bardual-\dualj{j})=0,\\
&\overline{\varphi}=\frac{1}{m(\Sang)}\sum_{j=1}^{N_{\Nang}}\omega_j\varphi^j.
\end{split}
\label{eq:dual_problem_angle_discretized}
\end{equation}
together with the final time condition $\dualj{j}(x,T)=\dualj{j}(x)\in C^{\infty}(X)$.

Note that for this dual problem satisfies
\begin{equation}
    \frac{d}{dt}\sum_{j=1}^{\Nang}\, \omj (\psij{j},\dualj{j})_{\BX}=0.
    \label{eqn:duality}
\end{equation}
This is an important property that allows us to phrase our error estimates in terms of the initial projection as well as the temporal growth in jumps across element interfaces.

\textbf{Approximation polynomials.} Throughout this article, we will utilize the \emph{Legendre polynomials} defined on a reference interval $[-1,1]$, which are defined by the recurrence relation
\begin{equation}
    \label{eq:legendre}
    \leg{0}(\xi) = 1,\quad \leg{1}(\xi) = \xi,\quad \leg{\dgorder}(\xi) = \frac{2\dgorder}{\dgorder}\leg{\dgdeg}(\xi)-\frac{\dgdeg}{\dgorder}\leg{\dgdeg-1}(\xi),\, \dgdeg \geq 1.
\end{equation}
These satisfy  $\int_{-1}^{1}\, \leg{\dgdeg}(\xi)\leg{j}(\xi)d\xi = \frac{2}{2\dgorder}\delta_{jk}$ with $\delta_{j\dgdeg}$ being the Kronecker delta 
\[\delta_{j\dgdeg}=\begin{cases}
1, \quad j=\dgdeg\\
0, \quad j\neq \dgdeg
\end{cases}.\]

The \emph{Radau polynomials} will also be useful.  The $\dgdeg$-th order left Radau polynomial, $\radp{\dgdeg},$ and the right Radau polynomial, $\radm{\dgdeg},$ on the reference element $[-1,1]$ are defined as
\begin{equation}
\label{eq:radau}
    \radpm{\dgdeg}(\xi) = \begin{cases}
            \leg{\dgdeg}(\xi)\pm \leg{\dgdeg-1}(\xi),\quad &\text{if}\;\dgdeg\geq 1,\\
            0, \quad &\text{if}\; k= 0.
            \end{cases}
\end{equation}
Note that $\radp{\dgdeg}(-1)=0$ and $\radm{\dgdeg}(1)=0$. In other words, the left (right) end point of the reference interval is a root of the left (right) Radau polynomial.

\textbf{Quadrature.} For the discretization in phase space-angle, we will need to define a quadrature that can integrate exactly polynomials of total degree less than or equal to $N$ ($|\alpha| = N$) \cite{freeden1998constructive} and denote it by:
\[
\barfh{f} = \quadf{N}{f_j},
\]
with $\omj$ denoting the quadrature weights and $f_j$ denoting the function $f$ evaluated at the quadrature points.  The weights are defined such that $\sum_{j=1}^N\, \omj = 1$.

For the quadrature, we have the following approximation result:
\begin{lemma}[Quadrature accuracy bound (\cite{hesse2006cubature}, Corollary)]
\label{lem:sn_approx}
Let $f\in H^s(\Sang),\, s>1$ (cf. \cite{freeden1998constructive})). Then 
\begin{equation}
    \left|\frac{1}{m(\Sang)}\int_{\Sang}f(\ang)\,d\ang - \quadf{N}{f(\angj{j})}\right|\leq C_s {N}^{-s}\Hs{f}{\Sang}{s}.
\end{equation}
$C_s$ is an universal constant depending only on $s$ and $m(\Sang)$ represents the Lebesgue measure $\Sang$. 
\end{lemma}

\textbf{Discretization.} Let $\Th$ be a spatial tessellation of the domain $\BX$ and $\elmt{}$ an element in the tessellation.  For simplicity of the analysis, we assume that $\elmt{}$ is a rectangular element, that is $\elmt{}$ is defined as $\elmt{}=I_{K,1}\otimes\dots\otimes I_{\elmt{},d}$, with $I_{\elmt{},j}$ being a one-dimensional interval. For simplicity, we denote $h_{\elmt{}}$ to be the size of $\elmt{}$. We also introduce the edges, $\edg{}$, and $\Eh = \cup\, \edg{}.$ 
For convenience, we introduce the shorthand notation
\[\int_{\BX} \, d\bx = \sum_{\elmt{}\in \Th}\, \int_{\elmt{}}  d\bx,\qquad \int_{\edg{}}  d\Eh =\sum_{e\in \Eh}\int_{\edg{}}  d\edg{}.\]

Given an integer $\dgdeg \geq 1$,  a polynomial approximation space for the tensor product polynomials, $\Qk{\dgdeg} = \Pk{\dgdeg}(I_{\elmt{},1}) \otimes \dots \otimes \Pk{\dgdeg}(I_{\elmt{},d})$, is defined as 
\begin{equation}
    \Vh{\dgdeg }=\{v\in L^2(\BX) : v\bigg|_{\elmt{}}\in \Qk{\dgdeg }(\elmt{}),\,\elmt{}\in \Th\}.
    \label{eqn:finite_element_space}
\end{equation}

On the edge, $\edg{}$, of element $\elmt{}$,  $\normal_x^{\pm}$ denotes the unit outward normal across $\elmt{}$. This allows for defining the jumps and averages on edges.  In order to do so, denote $g^{\pm}=\lim_{x\rightarrow \edg{}^{\pm}}g$. The jumps and averages across $\edg{}$ are then defined as 
\begin{equation}
    \dgjump{g}{\edg{}} = g^+\normal_{\edg{}}^+ - g^-\normal_{\edg{}}^-, \qquad  
    \dgavg{g}{\edg{}}=\frac{1}{2}\left(g^+ + g^-\right),
    \label{eqn:jumpavg}
\end{equation}
respectively.  Note that we will often drop the superscript for the normal vector and utilize $\normal_x^+ = -\normal_x^{-} = -\normal_x.$

\subsubsection{Estimates}

In this subsection we discuss some of the basic estimates that will be useful in the proof of superconvergence.  We note that we do not distinguish between constants and denote them by $C$.

We begin by discussing the following regularity results for a linear PDE system, which will be useful for the proof in the time-dependent case. 
\begin{lemma}[Regularity of the dual problem]\label{lem:regularity} 
Consider the following system of equations with periodic boundary conditions in $\bx$ for all $t\in[0,T]$:
\begin{equation}
\dpart{t}{}{\dualj{j}}+\mathbf{A}_j(\bx)\cdot\nabla_{\bx}\dualj{j}+\sigma\left[\bardual - \dualj{j}\right]=0,\, 1\leq j\leq \Nang.
        \label{eqn:general_dual_problem}
\end{equation}
where the given functions $\mathbf{A}_j\in W^{l+1,\infty}$ are periodic, and satisfy the divergence free constraint $\nabla_x \cdot\mathbf{A}_j(x)=0$. Then for all $t\in [0,T]$ there exists a constant independent of time $t$ satisfying:
\begin{equation}
\quadf{N_v}{\Hs{\dualj{j}(\cdot,t)}{\BX}{\ell}^2}\leq C\quadf{N_v}{\Hs{\dualj{j}(\cdot,T)}{\BX}{\ell}^2}.
\end{equation}
\end{lemma}
\begin{proof}
For the time dependent case we will assume $\sigma$ is constant.
    By testing the equation above by $\dualj{j}$ for each $j$, and using the periodic and divergence free properties of the $\mathbf{A_j}$'s, we have the following,
    \begin{equation}
    \begin{split}
        \frac{1}{2}\frac{d}{dt}\quadf{N_v}{\Hs{\dualj{j}(\cdot,t)}{\BX}{\ell}^2} = & -\sigma\quadf{N_v}{\int_{\BX}\, (\bardual-\dualj{}_j)\dualj{}_j\,d\bx} \\
        & = -\sigma\quadf{N_v}{\int_{\BX}(\bardual-\dualj{}_j)^2\,d\bx}\geq 0.
    \end{split}
    \end{equation}
    The conclusion follows. Now since we are considering the full Sobolev norm, we still need to estimate the $L^2$- weighted norms of the higher order spatial derivatives $\partial_{x_j}^\ell$,  to do so we apply $\partial_{x_j}$ to the equation above  and then we repeat the steps above.
\end{proof}

For the discretization, the following approximation properties for $\Vh{k}$, as well as some inverse inequalities \cite{ciarlet2002finite} will be necessary. Here, we denote by $\Pi^\dgdeg$ the usual $L^2-$projection onto $\Vh{\dgdeg}$ and define $\eta_h^{j}=\Pi^\dgdeg\psij{j}-\psij{j}$ as the \emph{projection error} as $\xi^j_h=\Pi^k\psij{j}-\psij{j}_h$ as the \emph{error in space}.  Then, $e^j_h=\psij{j}-\psij{j}_h=\xi^j_h-\eta^j_h$ and we have the following estimate for the projection error:
\begin{lemma}[Approximation properties\cite{ciarlet2002finite}]     \label{lem:approximation_properties}
There exists a constant $C>0$, such that for any $g\in H^{\dgorder}(\D)$,
\begin{equation}
\norma{\eta_h}_{L^2(\elmt{})}+h_{\elmt{}}\norma{\nabla_x\eta_h}_{L^2(\elmt{})}+h_{\elmt{}}^{1/2}\norma{\eta_h}_{L^2(\dK)} \leq Ch^{\dgorder}_{{\elmt{}}}\norma{\psij{}}_{\dgorder,{\elmt{}}},\qquad \forall {\elmt{}}\in\Th,\\
\label{eqn:projection_estimates}
\end{equation}
where the constant $C$ is independent of the mesh size $h_{{\elmt{}}}$, but depends on the polynomial degree, ${\elmt{}}$, and the regularity of $\psij{}$.
\end{lemma}

We will also make use of  the following inverse inequality:
\begin{lemma}[Inverse inequality \cite{ciarlet2002finite}]   \label{lem:inverse_inequality} There exists a constant $C>0$, such that for any $\psij{}\in \Pk{\dgdeg}({\elmt{}})$ the following holds:
\begin{equation}
    \norma{\nabla_x g}_{L^2(\elmt{})}\leq Ch^{-1}_{{\elmt{}}}\norma{g}_{L^2(\elmt{})},\qquad \norma{g}_{L^2(\dK)}\leq Ch_{\elmt{}}^{-1/2}\norma{g}_{L^2(\elmt{})},
    \label{eqn:inverse_inequalities}
\end{equation}
 where the constant $C$ is independent of the mesh size $h_{\elmt{}}$, but depends on polynomial degree $\dgdeg$ and the shape regularity of the mesh. 
\end{lemma}


\subsection{Discretization in angular space}

In angular space, we apply the discrete ordinates method $S_{N_{v}}$ for the 1D-steady-state case and $S_{\Nang}$ for the time-dependent model. For the steady state case this is the Gauss-Legendre quadrature points $\{v_j\}_{j=1}^{N_v}$ and weights $\{\omega_j\}_{j=1}^{N_v}$ in $[-1,1]$. By approximating the macroscopic density with the numerical quadrature, we obtain the $S_{N_v}$ system which seeks $\psi(x,v_j,t)\approx\psij{j}(x,t)$ satisfying
\begin{align}
\label{eq:steady}
    &\modelSteadyOne{j}{}\\
    & \barp(x,t) =\frac{1}{2}\sum_{j=1}^{N_v}\omega_j\psij{j}(x).
\end{align}

That is, equation \eqref{eq:1d_slab} is solved at a set of quadrature points. Analogously for the time-dependent problem \eqref{eq:hd_rt_reduced} using the $S_{\Nang}$ nodes $\{\angj{j}\}_{j=1}^{N_v}$ and weights $\{\omega_j\}_{j=1}^{N_v}$, we obtain the $S_{\Nang}$ system which seeks $\psi\vard{j}\approx\psij{j}(\bx,t)$ satisfying

\begin{subequations}
\label{eq:rte_sn}
\begin{align}
&\partial_t\psij{j}(\bx,t)+\angj{j}\cdot\nabla_{\bx}\psij{j}(\bx,t) =\sigma(\overline{\psi}(\bx,t)-\psij{j}(\bx,t))\label{eqn:md_rt_sn_disc}\\   
&\barf{\psi}(x)=\frac{1}{m(\Sang)}\quadf{\Nang}{\psij{j}}.    \label{eqn:disc_aver}
\end{align}
\end{subequations}
where $\bx\in \BX, \;1\leq j\leq \Nang.$

We note that the accuracy of the $S_N$ approximation is given by \cref{lem:sn_approx} applied to angular space.

\subsection{Spatial discretization via the DG method}

We apply high order upwind discontinuous Galerkin (DG) method to the $S_N$ system \eqref{eq:rte_sn}.  That is, 
we seek $\psijh{j} \in \Vh{\dgdeg}$ such that for all $\test\in \Vh{\dgdeg}$ and $1\leq j\leq \Nang$:
\begin{align}
    \left(\dpart{t}{}{\psijh{j}}, \test\right)_{\elmt{}} - \left(\psijh{j}\left(\angj{j}\cdot\nabla_x,\test\right)\right)_{\elmt{}} + &\left(\flux{\angj{j}\psijh{j}}{},\test\right)_{\dK}\notag\\
    & =\left(\sigma(\barfh{\psij{j}}-\psijh{j}),\test\right)_{\elmt{}}. \label{eqn:upwind_md}
\end{align}
Here, $\flux{\angj{j}\psijh{j}}{\dK}$ denotes the upwind numerical flux defined by :
\begin{equation}
    \flux{\angj{j}\psijh{j}}{}\bigg|_{\dK}=\left(\dgavg{\psijh{j}\Omega}{\dK}+\frac{|\Omega\cdot\normal_{\dK}|}{2}\dgjump{\psijh{j}}{\dK}\right)\normal_{\dK},
    \label{eqn:upwind_flux}
\end{equation}
where $\normal$ is the outward normal direction with respect to the element $\elmt{}$.


For the one-dimensional case, we consider the case of steady-state and denote the domain as $X=[x_L,x_R]$. For simplicity, we assume a uniform mesh given by $x_L=x_0<x_1<\dots<x_{\Nx}=x_R$ with mesh size $\dx=\frac{x_R-x_L}{\Nx}$. Additionally, we set $\angx{j}=\cos(\angj{j}).$ Then, the upwind DG scheme is: Find $\psijh{j} \in \Vh{\dgdeg},$ and $ j = 1,\dots, \NangO$ such that:
\begin{align}
\label{eq:1d_sn_dg}
-(\angx{j}\psijh{j},\dpart{x}{}{\test})_{\elmt{\Kindx}} & + (\sigt\psijh{j},\test )_{\elmt{\Kindx}}
+\flux{\angx{j}\psijh{j}}{}(x_{\Kindx})\test(x_p^-)-\flux{\angx{j}\psijh{j}}{}\test(x_{p-1}^+)\notag\\
&=(\sigma_s\barf{\psijh{}},\test)_{\elmt{\Kindx}},
\quad \forall 1\leq \Kindx \leq N_x,\;\test\in \Vh{\dgdeg}. 
\end{align}
where $K_\Kindx = (x_{\Kindx-1},x_{\Kindx})$ represents the $\Kindx-$th element.
Here, $x_{\Kindx}^{\pm}$ stands for the right and left limit of the function value.
For the one-dimensional case, the upwind flux is simply
\begin{equation}
\label{eq:1d_flux}
\flux{\angx{j}\psijh{j}} = 
\begin{cases}
    \angx{j}\psijh{j}(x_{\Kindx}^{-}),\quad  \angx{j} >0,\\
    \angx{j}\psijh{j}(x_{\Kindx}^{+}),\quad  \angx{j} \leq 0.
\end{cases}
\end{equation}
Particularly, given inflow boundary conditions,  $\psijh{j}(x^-_0)=\psij{j}(x_0,\angx{j})$ for $\angx{j}>0$ and $\psijh{j}(x^+_N)=\psij{j}(x_N,\angx{j})$ for $\angx{j}<0$.

We emphasize that the upwind DG discretization with polynomial order $\dgdeg\geq 1$ is proven to be asymptotically preserving for the radiative transport equation (RTE) \cite{Adams2002FastIM,guermond2010asymptotic,sheng2021uniform}. Hence, it can capture the correct diffusion limit of RTE without resolving the particle mean free path.

\subsection{Error estimates for the approximation}

For the error estimates, recall that we define $\eta^j_h=\Pi^k\psij{j}-\psij{j}$ to be the error between the projection of the exact solution and the exact solution and $\xi^j_h=\Pi^k\psij{j}-\psij{j}_h$ to be the error between the projection and the approximation. It is easy to see that $e^j_h=\psij{j}-\psij{j}_h=\xi^j_h-\eta^j_h$. From the approximation, Lemma \ref{lem:approximation_properties}, the following inequality holds: 
\begin{equation}
    \norma{\eta^j_h}\leq Ch^{\dgorder}\norma{\psij{j}}_{\dgorder,\BX},\,1\leq j\leq \Nang
\end{equation}
This will be used frequently.

We also have the following estimate:
\begin{theorem} \label{thm:dg_error_time_dep} Let $\dgdeg\geq 0$ and let $\psij{j},\, 1\leq j\leq \Nang$ be the exact solutions to \eqref{eq:rte_md} and assume that $\psij{j}\in C([0,T];H^{\dgorder}(\BX)),\, 1\leq j\leq \Nang$. If $\psijh{j}$ is an approximation obtained via Equation \eqref{eqn:upwind_md} with the numerical initial conditions $\psijh{j}(\cdot,0)=\Pi^\dgdeg\psij{j}_0$, then 
    \begin{equation}
        \left(\quadf{\Nang}{\norma{(\psij{j}-\psijh{j})(t)}^2_{L^2(\BX)}}\right)^{1/2}\leq C\dx^{\dgorder/2},\,\forall t\in [0,T].
    \end{equation}
    Here, the constant $C$ depends on the upper bound of $\quadf{\Nang}{\norma{\psij{j}}_{\dgorder,\BX}}$ .
\end{theorem}

Now we analyze the error coming from the angular part. 
\begin{theorem}\label{thm:DO_error_time_dep}Let $N$ be the accuracy of the $S_N$ quadrature, if $\psi\in C([0,T];L^2(\BX;H^{s}(\Sang)))$, then we have the following error estimate
    \begin{equation}
    \left(\sum_{j=1}^{N_v}\omega_j\norma{\psi(\cdot,\angj{j})-\psi^j}_{L^2(\BX)}^2\right)^{1/2}\leq D\Nang^{-s},
\end{equation}
That depends on $c_s$ is an universal constant depending only on $s$ and the upper bounds of $\norma{\psi}_{L^2(\BX;H^s(\Sang))}$.
\end{theorem}

For the proofs see the Appendix.

 By combining Theorems \ref{thm:dg_error_time_dep} and \cref{thm:DO_error_time_dep}, as well as the triangle inequality, we can conclude the following estimate:
\begin{theorem}\label{thm:total_error_bound} For polynomial degree $\dgdeg\geq 1$, if $\psij{}\in C([0,T];L^2(\BX;H^s(\Sang)))$ and each $\psij{j}\in C([0,T];H^{\dgorder}(\BX)),\, j = 1,\dots, \Nang$, the semi-discrete discrete-ordinates approximation via the Discontinuous Galerkin method for the time-dependent radiation transport equation,  \cref{eq:rte_md}, has the following error estimate
    \begin{equation}
    \left(\quadf{\Nang}{\norma{\psi(\cdot,\angj{j})-\psi_h^j}_{L^2(\BX)}^2}\right)^{1/2}\leq Ch^{\dgorder/2}+D {\Nang}^{-s},
    \end{equation}
where $C$ and $D$ depend on upper bounds of $\quadf{\Nang}{\norma{\psij{j}}_{\dgorder,\BX}},\, \norma{\psij{}}_{L^2(\BX;H^s(\Sang))}$ respectively. 
\end{theorem}
  
\section{Superconvergence and extraction\label{sec:superconvergence-filter}}

In this section we discuss the superconvergence of the discontinuous Galerkin approximation to Equation \cref{eq:rte_md}.  We begin by introducint the superconvergent points in the one-dimensional steady-state case.  We then proceed to discuss the multi-dimensional time-dependent case and the underlying superconvergence in the negative-order norm. In Section \cref{sec:SIAC} we show how to extract this information via the Smoothness-Increasing Accuracy-Conserving (SIAC) filter.

In our superconvergence analysis, we assume periodic or zero inflow boundary conditions as well as a constant scattering cross section, $\sigs$. Note that though our analysis is restricted to constant scattering cross sections,  we numerically observe improved accuracy after post-processing for general cross sections. 

\subsection{Superconvergence of the steady state problem for a 1D slab geometry}

In this section, we discuss the superconvergent points of the one-dimensional steady-state model \eqref{eq:steady}.

The discontinuous Galerkin approximation given in  \cref{eq:1d_sn_dg} in terms of the local coordinate mapping becomes: Find $\psijh{j}\in\Vh{\dgdeg}$ such that
\begin{align}
\label{eq:1ddg}
    0 =  & -(\angx{j}\psijh{j},(\test)')_{\I} +  \left(\flux{\angx{j}\psijh{j}}{\Kindx}\test(1) -\flux{\angx{j}\psijh{j}}{\Kindx-1}\test(-1)\right)\\
    & \qquad \qquad \qquad   + \frac{\dx}{2}\left[(\sigt\psijh{j},\test)_{\I} - (\sigs\barfh{\Psi},\test)_{\I}\right] \qquad \forall \test \in \Pk{\dgdeg}.\notag
\end{align}
The mapping to the reference element, $\I = [-1,1]$, is defined by $\map = \frac{2}{\dx}\left(x-\bxn{\Kindx}\right)-1,$ where $\bxn{\Kindx} = \frac{1}{2}(x_{\Kindx}+x_{\Kindx+1}).$  
We then have the following theorem: 

\begin{theorem}
\label{thm:steady_state}
\textbf{(Superconvergence at Radau points for steady state problem)}    
Suppose $\psij{j}$ is a Lipschitz continuous function belonging to $C^{\infty}([x_L,x_R])$,  with $\sigs$ and $\sigt$ being constant, and $\psijh{j}\in \Vh{\dgdeg}$ being the approximation to the $S_N$ system \eqref{eq:steady} and its DG approximation \eqref{eq:1d_sn_dg}. Define the error function as $\ehj{j}(x,t) = \psij{j}(x,t)-\psijh{j}(x,t)$. When the mesh size is sufficiently small, the following statements hold.
\begin{enumerate}
    \item When $\angx{j}>0$, $\ehj{j}(x_p^{-},t)=\Oh(\dx^{2\dgdeg+2})$ for $1\leq p\leq \NangO$. When $\angx{j}<0$, $\ehj{j}(x_{\Kindx-1}^{+},t)=\Oh(\dx^{2\dgdeg+2})$ for $1\leq p\leq \NangO$.
    \item Let the roots of the $(\dgorder)$-th order right and left Radau polynomial $R^{\pm}_{\dgorder}$ be $\{\nu^+_\ell\}_{\ell=1}^{\dgdeg}\bigcup\{1\}$ and $\{-1\}\bigcup\{\nu^-_\ell\}_{\ell=1}^{\dgdeg}$. For the interior roots, when $\angx{j}>0$, $$\ehj{j}\left(\bxn{\Kindx}+\frac{2}{\dx}\nu^+_\ell\right)=\Oh(\dx^{\dgdeg+2})$$ for $1\leq \ell\leq \dgdeg$. And when $\angx{j}<0$, $$\ehj{j}\left(\bxn{\Kindx}+\frac{2}{\dx}\nu^-_\ell\right)=\Oh(\dx^{\dgdeg+2})$$ for $1\leq \ell\leq \dgdeg$. 
\end{enumerate}
\end{theorem}

Notice that when $\sigs=\sigt=0$, the $S_N$ system \eqref{eq:1d_slab} degenerates to a $\NangO$ decoupled linear advection equation. As a result, Theorem \ref{thm:steady_state} implies that the upwind DG method has $2\dgdeg+2$ order of accuracy at downwind edges. In other words, for linear advection, the global superconvergence order for the downwind edge of each element is $2\dgdeg+2$ instead of $2\dgorder$ proved in \cite{adjerid2002posteriori}. Our proof, which will be outlined in Sec. \ref{sec:proof_outline}, can be seen as an extension of \cite{adjerid2002posteriori}. 
\subsection{Superconvergence for time dependent case}

The time-dependent case considers multi-dimensional approximations to \cref{eq:rte_sn}.  For this analysis, we will need to establish a bound on a weighted negative-order norm of the error as we utilized the $S_N$ system.

\begin{theorem} \label{thm:negative_order_estimate} Let $\dgdeg\geq 0$,  Let $\psij{j}$, $1\leq j\leq \Nang$ be the exact solutions to \eqref{eq:rte_md} and assume that $\psij{j}\in C([0,T];H^{\dgdeg+2}(X))$, $1\leq j\leq \Nang$. If $\psijh{j}$ is an approximation obtained via the DG formulation, \cref{eqn:upwind_md}, with the numerical initial conditions $\psijh{j}(\cdot,0)=\Pi^\dgdeg\psij{j}_0(\cdot)$, then 
\begin{equation}
    \norma{\Psi-\Psi_h}_{-(\dgorder),\BX}\leq C \dx^{2\dgdeg+1/2}.
    \label{eqn:negative_estimate}
\end{equation}
Where $\Psi=(\psi^j)_{j=1}^{N_{\ang}}$ and $\Psi_h=(\psi_h^j)_{j=1}^{N_{\ang}}$ Here, $\norma{\cdot}_{-(\dgorder),\BX}$ represents a weighted negative-order norm defined in \cref{eqn:negative-order-norm} and $C$ depends on upper bounds of $\norma{\psij{}}_{\dgdeg+2,\BX}$
\end{theorem}

For the case with constant scattering, we note that this estimate also holds for the divided-differences of the error.


The proof of the negative order norm is given in Section~\ref{sec:num-accuracy-td}. 
It relies on the following three estimates, which follow the same structure as in \cite{CLSS2}, 
with the additional consideration of the collision operator.

\begin{lemma}[Projection Estimate]\label{lem:project} Assume that the same assumptions hold as in Theorem \ref{thm:negative_order_estimate}, then, defining $\Theta_M = -\quadf{\Nang}{(\eta^j_h,\dualj{}(0))},$ the following estimate holds for the projection error
\begin{equation}
    \left|\Theta_M\right|\leq Ch^{2\dgdeg+2}\sqrt{\quadf{\Nang}{\norma{\dualj{j}(0)}_{\dgorder,\BX}^2}},
\end{equation}
where $C$ depends on $\left(\quadf{\Nang}{\norma{\psij{j}_0}_{\dgorder,\BX}^2}\right)^{1/2}$.
\end{lemma}

\begin{lemma}[Residual estimate]\label{lem:residual} Let $\chi^j=\Pi^\dgdeg\dualj{j}$  for each $1\leq j\leq {\Nang}$. Assume that the same assumptions hold as in Theorem \ref{thm:negative_order_estimate}, then, defining  
\[
\Theta_N=-\int_0^T\quadf{\Nang}{\left[\left((\psijh{j})_t,\dualj{j}-\chi^j\right)(s)+B_h\left((\psijh{j},\dualj{j}-\chi^j;\angj{j})\right)(s)-\sigma\left(\barfh{\psij{}}-\psijh{j},\dualj{j}-\chi^j\right)(s)\right]}\,ds,\]
where $\chi^j$ in $V_h$, we have the following: 
\begin{equation}
    |\Theta_N|\leq C\dx^{2\dgdeg+1/2}\left[\int_0^T\quadf{\Nang}{\norma{\dualj{j}}_{\dgorder,\BX}^2}\,ds\right]^{1/2}
\end{equation}
where $C$ depends on the upper bounds of $\norma{\psijh{j}}_{\dgdeg+2,\BX}$.
\end{lemma}

\begin{lemma}[Consistency estimate]\label{lem:consistency} Assume that the same assumptions hold as in Theorem \ref{thm:negative_order_estimate} and define 
\[
\Theta_C=-\int_0^T\quadf {\Nang}{\left[\left(\psijh{j},\dualj{j}_t\right)-B_h\left(\psijh{j},\dualj{j};\angj{j}\right)+\sigma\left(\barfh{\psij{}}-\psijh{j},\dualj{j}\right)\right]}\, ds.
\]  One can show that $\Theta_C$ satisfies
    \begin{equation}
        \Theta_C=0.
    \end{equation}
\end{lemma}


\section{Superconvergence analysis \label{sec:proof_outline}}
In this section, we outline main steps in the proofs of our superconvergence results. 

\subsection{Analysis of superconvergence for the one-dimensional steady-state problem\label{sec:steady_state_proof}}

Here, we sketch key steps to prove Theorem \ref{thm:steady_state}. Throughout this subsection, we restrict the analysis of the DG solution $\psijh{j}$ and its error function to the $\Kindx-$th element,  $\elmt{\Kindx} = [x_{\Kindx-1},x_{\Kindx}].$ The error on $\elmt{\Kindx}$ is denoted as $\ehj{j}\bigg|_{\elmt{\Kindx}} = \ejn{j}{\Kindx},\, j=1,\dots,{\NangO}$ and $e_{\elmt{}}=(\ejn{1}{\Kindx},\ejn{2}{\Kindx},\dots,\ejn{{\NangO}}{\Kindx})$.

We use the same framework for the superconvergence analysis of the linear advection equation as used by Adjerid et al. \cite{adjerid2002posteriori} and match the coefficients of the polynomial expansion of the error equation. A key step in \cite{adjerid2002posteriori} is to utilize the theory for the Runge-Kutta method by viewing the one-dimensional linear advection equation as an ODE. However, unlike the one-direction linear advection equation, the RTE equation involves both left- and right-going particles coupled through an integral term. As a result, the ODE argument in \cite{adjerid2002posteriori} is not applicable for kinetic equations. To bypass this difficulty, we utilize element-by-element mathematical induction.  

Before proceeding with the analysis, we note that the exact solution on element $\elmt{\Kindx}$ can be expressed as a Taylor series centered around the element center, $\bxn{\Kindx}$: 
\begin{equation}
\label{eq:taylorpsi}
    \psij{j}(x,\angx{j})\bigg|_{\elmt{\Kindx}}  = \sum_{m=0}^\infty\, \underbrace{\frac{1}{m!}\dpart{x}{m}{\psij{j}}(\bxn{\Kindx},\angx{j})}_{\apj{j}{\Kindx}{m}}\left(\frac{\dx}{2}\right)^m\map^m,
\end{equation}
where $\map$ is the local coordinate mapping.
Noting the relation between monomials, $\map^m$, and Legendre polynomials, 
\[\map^m = \sum_{s=0}^{\lfloor\frac{m}{2}\rfloor}\, \beta_{m-2s,s}\leg{m-2s}(\map) :=P_m(\map),
\] 
where
\[
\beta_{q,s} = (2q-4s+1)\frac{(2\lfloor\frac{q+2s}{2}\rfloor)!!}{(2s)!!}\frac{(2\lfloor\frac{q+2s+1}{2}\rfloor-1)!!}{(2(q+s)+1)!!},\qquad m \mathbb{N}_+.
\]
This can be written as 
\begin{equation}
\label{eq:taylorP}
    \psij{j}(x,\angx{j})\bigg|_{\elmt{\Kindx}}  = \sum_{m=0}^\infty\, \apj{j}{\Kindx}{m}P_{m}(\map)\left(\frac{\dx}{2}\right)^m.
\end{equation}
This relation can also be seen by using an $L^2-$projection of $\psij{j}(x,\angx{j})$ onto the Legendre polynomials:
\begin{equation}
\label{eq:Legendrepsi}
    \psij{j}(x,\angx{j})\bigg|_{\elmt{\Kindx}} = \sum_{q=0}^\infty\, c_{\Kindx,q}^j L_q(\map_\Kindx),\qquad c_{\Kindx,q}^j = \sum_{s=0}^{\infty}\, 
\apj{j}{\Kindx}{q+2s}\left(\frac{\dx}{2}\right)^{q+2s}\beta_{q,s}.
\end{equation}  

\noindent\textbf{Derivation of the error equation on the reference element.} Denote the error on element $\Kindx$ as
\[
\ehj{j} = \psij{j} - \psijh{j},\quad j = 1,\dots,{\NangO}.
\]
Then, the error equation on the reference element is given by
\begin{align}
\label{eq:1ddg}
    0 =  & -(\angx{j}\ejn{j}{\Kindx},\tau')_{\I} +  \flux{\angx{j}\ehj{j}}{\Kindx}\tau(1) -\flux{\angx{j}\ehj{j}}{\Kindx-1}\tau(-1) + \frac{\dx}{2}\left((\sigt\ejn{j}{\Kindx},\tau)_{\I} -(\sigs\barf{\ejn{}{\Kindx}},\tau)_{\I}\right) . 
\end{align}
Using an upwind flux, the error equation then becomes
\begin{align}
\label{eq:1ddg}
    0 =  & -(\angx{j}\ejn{j}{\Kindx},\tau')_{\I} + \frac{\dx}{2}\left((\sigt\ejn{j}{\Kindx},\tau)_{\I} -(\sigs\barf{\ejn{}{\Kindx}},\tau)_{\I}\right) +
    \begin{cases}
    \angx{j}\ejn{j}{\Kindx}(1)\tau(1)-\angx{j}\ejn{j}{{\Kindx-1}}(1)\tau(-1)  \qquad &\angx{j}>0\\
    \angx{j}\ejn{j}{\Kindx+1}(-1)\tau(1)-\angx{j}\ejn{j}{\Kindx}(-1)\tau(-1)  \qquad &\angx{j}<0
    \end{cases} . \notag
\end{align}

\vspace{0.5cm}

\noindent\textbf{Polynomial expansion of the error function on the reference element as a series of $\Delta x$.}
Using assumptions in Theorem \ref{thm:steady_state}, and the Taylor series expansion \cref{eq:taylorpsi}, the error in terms of powers of $\dx$ is
\begin{equation}
\label{eq:TaylorErr}
    \ehj{j}(x,\angx{j})\bigg|_{\elmt{\Kindx}} = \sum_{m=\dgorder}^\infty\, \apj{j}{\Kindx}{m}\Qpm{}{m}(\map)\left(\frac{\dx}{2}\right)^m,
\end{equation}
where $\Qpm{}{m}(\map)$ is a polynomial of degree $m.$  We will show that it is defined as
\begin{equation}
    \label{eq:Qpm}
    \Qpm{}{m}(\map) = \begin{cases}
        \widetilde{\beta}_{\dgorder}^{\pm}\radpm{\dgorder}(\map),\qquad m=\dgorder\\
        \sum_{s=0}^{\lceil\frac{m-k}{2}\rceil-1}\, \beta_{m-2s,s}\leg{m-2s}(\map),\qquad m\geq\dgdeg+2
    \end{cases}.
\end{equation}
Notice that $\Qpm{}{m}(\map)$ is written in terms of the local coordinates and does not depend on the element itself. Further, note that naively applying the expansion in \cref{eq:taylorP} would lead to 
\[
\ejn{j}{n,\dgorder} = \apj{j}{\Kindx}{m}P_m(\map)\left(\frac{\dx}{2}\right)^m.
\]
We will see that this is not the case.

Grouping the error equation in terms of powers of $\dx$ leads to the system:
\begin{equation}
\label{eq:dxEqn}
\begin{split}
    0 =& -(\angx{j}\ejn{j}{\Kindx,\dgorder},\tau')_{\I}
    + \flux{\angx{j}\ehj{j}}{\Kindx}\tau(1) +\flux{\angx{j}\ehj{j}}{\Kindx-1}\tau(-1)  + \begin{cases}
  0 &\qquad m = \dgorder\\ 
   (\sigt\ejn{j}{\Kindx,m-1},\tau)_{\I} - (\sigs\barf{\ejn{}{n,m-1}},\tau)_{\I} &\qquad m\geq \dgdeg+2 \notag
\end{cases},
\end{split}
\end{equation}
for $\tau \in \Pk{\dgdeg}$. 
In terms of the coefficients, this becomes:
\begin{align}
\label{eq:ErrEq}
0 = &-\angx{j}\apj{j}{\Kindx}{\dgorder}(\Qpm{}{\dgorder},\tau')_{\I}+\\
&\qquad \qquad +\angx{j}\begin{cases}
    \left(\apj{j}{\Kindx}{\dgorder}\tau(1)-\apj{j}{\Kindx-1}{\dgorder}\tau(-1)\right)\Qpm{}{\dgorder}(1),\quad &\angx{j}>0\\
    \left(\apj{j}{\Kindx+1}{\dgorder}\tau(1)-\apj{j}{\Kindx}{\dgorder}\tau(-1)\right)\Qpm{}{\dgorder}(-1),\quad &\angx{j}<0
\end{cases},\qquad m=\dgorder\\
0 = &-\angx{j}\apj{j}{\Kindx}{m}(\Qpm{}{m},\tau')_{\I} + (\Qpm{}{m-1},\tau)_{\I}\left(\sigt\apj{j}{\Kindx}{m-1}-\sigs\overline{\apj{}{\Kindx}{m-1}}\right)\\
&\qquad \qquad +\angx{j}\begin{cases}
    \left(\apj{j}{\Kindx}{m}\tau(1)-\apj{j}{\Kindx-1}{m}\tau(-1)\right)\Qpm{}{m}(1),\quad &\angx{j}>0\\
    \left(\apj{j}{\Kindx+1}{m}\tau(1)-\apj{j}{\Kindx}{m}\tau(-1)\right)\Qpm{}{m}(-1),\quad &\angx{j}<0
\end{cases},\qquad m\geq\dgdeg+2.\notag
\end{align}
for all $\tau \in \Pk{\dgdeg}$.  Note that for the $S_N$ system
\[
\overline{\apj{}{\Kindx}{m}} =\quadf{\NangO}{\apj{j}{\Kindx}{m}}.
\]

\noindent\textbf{Proof that $\Qpm{}{\dgorder}(\map)=\widetilde{\beta_{\dgorder}}\radpm{\dgorder}(\map)$.}
To prove that the leading order error term is a Radau polynomial, we set $\tau(\map) = \leg{q}(\map),\, q=0,\dots,p$ in \cref{eq:ErrEq}.  Then, the leading order error is:
\begin{align}
\label{eq:leadorderErr}
0 = &-\angx{j}\apj{j}{\Kindx}{\dgorder}(\Qpm{}{\dgorder}(\map),\leg{q}'(\map))_{\I} +\angx{j}\begin{cases}
    \left(\apj{j}{\Kindx}{\dgorder}-(-1)^q\apj{j}{\Kindx-1}{\dgorder}\right)\Qpm{}{\dgorder}(1),\quad &\angx{j}>0\\
    \left(\apj{j}{\Kindx+1}{\dgorder}-(-1)^q\apj{j}{\Kindx}{\dgorder}\right)\Qpm{}{\dgorder}(-1),\quad &\angx{j}<0
\end{cases}
\end{align}
For $q=0$ this simplifies to
\begin{align}
\label{eq:qzero}
0 =
&\angx{j}\begin{cases}
    \left(\apj{j}{\Kindx}{\dgorder}-\apj{j}{\Kindx-1}{\dgorder}\right)\Qpm{}{\dgorder}(1),\quad &\angx{j}>0\\
    \left(\apj{j}{\Kindx+1}{\dgorder}-\apj{j}{\Kindx}{\dgorder}\right)\Qpm{}{\dgorder}(-1),\quad &\angx{j}<0
\end{cases}
\end{align}

This implies that, for $\angx{j}>0$, either $\apj{j}{\Kindx}{\dgorder}-\apj{j}{\Kindx-1}{\dgorder} = 0,\, \Kindx = 1,\dots,\Nx,\, $ or $\Qpm{}{k+1}(1)=0$ and for $\angx{j}<0$ either $\apj{j}{\Kindx+1}{\dgorder}-\apj{j}{\Kindx}{\dgorder} = 0,\, \Kindx = 1,\dots,\Nx,\, $ or $\Qpm{}{k+1}(-1)=0$. If the difference in the coefficients is zero, one can prove by induction that $\apj{j}{\Kindx}{m}=0$ for all $m,\, \Kindx$ since we assume the inflow boundary is exact. This implies that the error is identically zero, which can only occur if our exact solution is a polynomial in the approximation space. 

Next, consider $q = 1,\dots,k:$
\[
0 = -\angx{j}\apj{j}{\Kindx}{\dgorder}(\Qpm{}{\dgorder}(\map),\leg{q}'(\map))_{\I}.
\]
Noting that 
\[
(\Qpm{}{\dgorder}(\map),\leg{q}'(\map))_{\I} = \sum_{s=0}^{\lfloor\frac{q-1}{2}\rfloor}\, (2(q-2s)-1)(\Qpm{}{\dgorder}(\map),\leg{q-(2s+1)}(\map))_{\I},\quad q = 1, ..., k.
\]
leads to the conclusion that the quantity on the right is always zero as $q-(2s+1) \leq k-1.$ Hence $\Qpm{}{\dgorder}(\map)$ is orthogonal to polynomials of degree $\leq \dgdeg-1$ and, combining the two results above, leads to  
\[
\Qpm{}{\dgorder}(\map) = \begin{cases}
        \widetilde{\beta}_{\dgorder}^{-}\radm{\dgorder}(\map),\quad & \angx{j}>0\\
        \widetilde{\beta}_{\dgorder}^{+}\radp{\dgorder}(\map),\quad & \angx{j}<0
\end{cases}.
\]
Combining (i) the local Taylor series expansions for both the exact solution and the DG approximation; (ii) the relation between the monomials and the Legendre polynomials, \cref{eq:taylorP}; and (iii) Galerkin orthogonality leads to the expression, for $m\geq\dgdeg+2$,
\begin{align*}
    \Qpm{}{m}(\map) = & \left(\sum_{s=0}^{\lfloor\frac{m}{2}\rfloor}\, \beta_{m-2s,s}\leg{m-2s}(\map)\right) - \left(\sum_{s=\lceil\frac{m-k}{2}\rceil}^{\lfloor\frac{m}{2}\rfloor}\, \beta_{m-2s,s}\leg{m-2s}(\map)\right),\\
    &\qquad \qquad =\left(\sum_{s=0}^{\lceil\frac{m-k}{2}\rceil-1}\, \beta_{m-2s,s}\leg{m-2s}(\map)\right)\qquad m=\dgdeg+2,\dots,2\dgdeg+1.
\end{align*}
Thus proving the second part of our theorem.

\noindent\textbf{Orthogonality of $\Qpm{}{m}$ for $m \geq \dgdeg+2$.} Consider the case where $m \geq \dgdeg+2$ and $\tau(\map) = \leg{q}(\map),\, q = 0,\dots,k$.  In this case, the error equation is given by
\begin{align*}
0 = -\angx{j}\apj{j}{\Kindx}{m}(\Qpm{}{m},\leg{q}'(\map))_{\I} + &(\Qpm{}{m-1},\leg{q}(\map))_{\I}\left(\sigt\apj{j}{\Kindx}{m-1}-\sigs\overline{\apj{}{\Kindx}{m-1}}\right)\\
&\qquad \qquad +\angx{j}\begin{cases}
    \left(\apj{j}{\Kindx}{m}-(-1)^q\apj{j}{\Kindx-1}{m}\right)\Qpm{}{m}(1),\quad &\angx{j}>0\\
    \left(\apj{j}{\Kindx+1}{m}-(-1)^q\apj{j}{\Kindx}{m}\right)\Qpm{}{m}(-1),\quad &\angx{j}<0
\end{cases}. 
\end{align*}

Rearranging the to obtain a relation for the $Q_m(1)$ gives
\begin{align*}
\angx{j}\begin{cases}
    \left(\apj{j}{\Kindx}{m}-(-1)^q\apj{j}{\Kindx-1}{m}\right)\Qpm{}{m}(1),\quad &\angx{j}>0\\
    \left(\apj{j}{\Kindx+1}{m}-(-1)^q\apj{j}{\Kindx}{m}\right)\Qpm{}{m}(-1),\quad &\angx{j}<0
\end{cases}=\angx{j}\apj{j}{\Kindx}{m}(\Qpm{}{m},\leg{q}'(\map))_{\I} - &(\Qpm{}{m-1},\leg{q}(\map))_{\I}\left(\sigt\apj{j}{\Kindx}{m-1}-\sigs\overline{\apj{}{\Kindx}{m-1}}\right).
\end{align*}
By Galerin orthogonality, the right side is zero for $q=0,1,\dots,\dgdeg$.  Hence the first non-vanishing term at $\map=1$ is $Q_{2(\dgorder)}$ and our theorem \cref{thm:steady_state} is proven.

\subsection{Analysis of superconvergence for time dependent case}

In this section, we outline the proof of Theorem \cref{thm:negative_order_estimate}:
\begin{proof}
    Let $\Phi\in (C_0^{\infty}(\BX))^N$, Then by \eqref{eqn:duality} and the dual problem definition, 
    \begin{align*}
        \quadf{\Nang}{(\ehj{j}(T),\phi^j)}&=\quadf{\Nang}{\left(\ehj{j}(T),\dualj{j}(T)\right)}\\
                                             &=\quadf{\Nang}{\left[\left(\psi^j(T),\dualj{j}(T)\right)-\left(\psijh{j}(T),\dualj{j}(T)\right)\right]}\\
                                             &=\quadf{\Nang}{\left[\left(\psi^j(0),\dualj{j}(0)\right)-\left(\psijh{j}(0),\dualj{j}(0)\right)-\int_0^T\, \frac{d}{dt}\left(\psijh{j},\dualj{j}\right)\,ds\right]}\\
                                             &=\quadf{\Nang}{\left[\left(\psi^j(0)-\psijh{j}(0),\dualj{j}(0)\right)-\int_0^T\, \left(\left((\psijh{j})_t,\dualj{j}\right)-\left(\psijh{j},(\dualj{j})_t\right)\right)\, ds\right]},                
    \end{align*}
    Notice that for any $\chi^j$ in $V_h$, 
    \begin{align*}
        \int_0^T\quadf{\Nang}{\left((\psijh{j})_t,\dualj{j}\right)}\, ds &= \int_0^T\quadf{\Nang}{\left[\left((\psijh{j})_t,\dualj{j}-\chi^j\right)+\left((\psijh{j})_t,\chi^j\right)\right]}\,ds\\
        &=\int_0^T\quadf{\Nang}{\left[((\psijh{j})_t,\dualj{j}-\chi^j)-B_h(\psijh{j},\chi^j;\omj)+\sigma(\barfh{\psij{}}-\psijh{j},\chi^j)\right]}\,ds\\
        &=\int_0^T\quadf{\Nang}{\left[(\psijh{j})_t,\dualj{j}-\chi^j)\,d\tau-B_h(\psijh{j},\chi^j;\omj)+(\barfh{\psij{}}-\psijh{j},\chi^j)\right]}\, ds\\
        &=\int_0^T\quadf{\Nang}{\left[((\psijh{j})_t,\dualj{j}-\chi^j)+B_h(\psijh{j},(\dualj{j}-\chi^j);\omj)-\sigma(\barfh{\psij{}}-\psijh{j},(\dualj{j}-\chi^j))\right]}\, ds\\
        &+\int_{0}^T\quadf{\Nang}{\left[-B_h(\psijh{j},\dualj{j};\omj)+(\barf{\psijh{}}-\psijh{j},\dualj{j})\right]}\, ds
    \end{align*}
    This allows for the error to be written as
    \begin{equation}
        \quadf{\Nang}{(\ehj{j}(T),\phi^j)}=\Theta_M+\Theta_N+\Theta_C
    \end{equation}
    where
    \begin{align*}
        \Theta_M &= -\quadf{\Nang}{(\eta^j_h,\varphi(0))},\\
        \Theta_N &= -\int_0^T\, \quadf{\Nang}{\left[\left((\psijh{j})_t,\dualj{j}-\chi^j\right)+B_h\left(\psijh{j},\dualj{j}-\chi^j;\omj\right)-\sigma\left(\barfh{\psij{}}-\psijh{j},\dualj{j}-\chi^j\right)\right]}\,ds\\
        \Theta_C &= -\int_0^T\, \quadf{\Nang}{\left[\left(\psijh{j},\dualj{j}_t\right)-B_h\left(\psijh{j},\dualj{j};\omj\right)+\sigma\left(\barfh{\psij{}}-\psijh{j},\dualj{j})\right)\right]}\,ds
    \end{align*}
    $\Theta_M$, $\Theta_N$ and $\Theta_C$ are respectively the \emph{projection, residual and consistency terms}.
    Using Lemmas \ref{lem:project}, \ref{lem:residual}, and \ref{lem:consistency} together with the dual estimate, \cref{eqn:general_dual_problem} gives our desired estimate.
\end{proof}

\section{Extracting superconvergence using the SIAC filter\label{sec:SIAC}}

Now that we have proven that higher-order accuracy exists in the negative-order norm, we show how to extract that information via the Smoothness-Increasing Accuracy-Conserving (SIAC) filter.

To illustrate the ability of SIAC to perform on a given data set, it is useful to outline how SIAC works for general data as well as through the error estimates.  

Given $\psijh{j},\, j = 1,\dots,\Nang,$ superconvergence can be extracted through convolving with a specially designed kernel, $\Kernel{}{H}{\cdot}$: 
\begin{equation}
\label{eq:postp}
(\psijh{j})^*(\bx) =  \Kernel{}{H}{\bx}\star \psijh{j}(\bx) = \int_\mathbb{R} \Kernel{}{H}{\bx-\by}\psijh{j}(\by) d\by.
\end{equation}
where $H$ represents the kernel scaling, in this case the uniform mesh size. 

We show the reliance on the negative-order norm by decomposing the filtered error into a term that only depends on the number of moments and a term that relies on the error in the negative-order norm:
\begin{equation}
\label{eq:errordecomp}
\| \psij{j} - (\psijh{j})^* \| \leq \underbrace{\|\psij{j}  - \Kernel{}{H}{}\star \psij{j}\|}_{\text{Moments}}+\underbrace{\| \Kernel{}{H}{}\star(\psij{j}-\psijh{j}) \|}_{\text{relies on negative-order norm}} \leq \mathcal{O}(H^{r+1}) +  \mathcal{O}(h^{s}) 
\end{equation}
where $\|\cdot\|$ is some norm. Here, the choice of $r$ is the number of moments the filter is designed to capture and $s$ is the order of accuracy of the approximation.  In this article, the ability to bound the $L^2-$norm by the negative-order norm is utilized and $r=2\dgdeg+1.$

The success of the filter relies on the following results. 
\begin{theorem}[Bramble and Schatz \cite{bramble1977higher}] \label{thm:bramble_schatz} Let $k\geq 0$. For $T>0$, let $\Psi=(\psi^n)_{n=1}^{\Nang}$ be the exact solution to problem \eqref{eqn:disc_aver}, satisfying $\psi^j\in C([0,T];H^{k+1}(\BX))$, $1\leq j\leq \Nang$. Let $\Omega_0+2 supp(K_h^{2(k+1),k+1}(\bx))\subset\subset \BX$ and $\Psi_h=(\psi^j_h)_{j=1}^{\Nang}$ then

\begin{align}
    &\left(
      \sum_{j=1}^{\Nang}\omega_j
      \norma{ \psi^j(T) -K_h^{2(k+1),k+1}\star \psi^j_h(T)}_{L^2(\Omega_0)}^2
     \right)^{1/2} \notag \\
    &\quad \leq 
      \frac{h^{2k+2}}{(2k+2)!}
      \left(
        \sum_{j=1}^{\Nang}\omega_j
        |\psi^j|_{2k+2,\BX}^2
      \right)^{1/2}
      + C_{\mathrm{P}}
      \sum_{|\lambda|\leq k+1} 
        \norma{\partial_h^{\lambda}(\Psi-\Psi_h)}_{-(k+1),\BX}.
    \label{eqn:bramble_schatz}
\end{align}
where $C_{\mathrm{P}}$ depends solely on $\Omega_0$, $\BX$, $k$, and it is independent of $h$.
\end{theorem}

In \eqref{eqn:bramble_schatz}, we used the notation of the divided differences, which are defined as 
\begin{equation}
    \partial_{h_{x_i}}f(\bx)=\frac{1}{h_{x_i}}\left(f(\bx+\frac{1}{2}h_{x_i}\mathbf{e}_i)-f(\bx-\frac{1}{2}h_{x_i}\mathbf{e}_i)\right).\label{eq:divided_difference}
\end{equation}
where $\mathbf{e}_i$ is the unit multi-index whose $i$-th component is $1$ and all others $0$. 

For any multi-index $\lambda=(\lambda_{1},\lambda_{2},\lambda_{3})$ we set the $\lambda$-th order difference quotient to be 
\begin{equation}
    \partial_h^{\lambda}f(\bx)=(\partial^{\lambda_1}_{h_{x_1}}\partial^{\lambda_2}_{h_{x_2}}\partial^{\lambda_3}_{h_{x_3}})f(\bx).
    \label{eq:ho_divided_difference}
\end{equation}

\subsection{SIAC formulation.}

The Smoothness-Increasing Accuracy-Conserving (SIAC) kernel is comprised of $r+1$ (scaled) function translates of a given function,
\begin{equation}
\label{eq:genkernel}
\Kernel{}{}{\cdot} = \sum_{\gamma=1}^{r+1}\, {c_\gamma}\psi_{\knots{\gamma}}(\cdot),
\end{equation}
in this article, central B-Splines are used.  ${ B_{\text{\tiny T},n}}$ represents $n^{th}$-order central B-spline with knot sequence, $\text{\textbf T}$ and smoothness $n-2$.  The scaling, $H$, is generally tied to the mesh size.  The central
B-splines are defined through the relations
\[
\boldsymbol{B}_{T,1} = \chi_{\left[-\frac 12,\frac 12\right)},\qquad
\boldsymbol{B}_{T,n} = \boldsymbol{B}_{T,n-1} \star \boldsymbol{B}_{T,1},
\]
 where ${\bf T}_{n}$ represents a knot matrix for the $n^{th}$ order spline (i.e. B-spline breaks) \cite{XLiOne}.  For a symmetric kernel of $2\dgdeg+1$ B-splines, the general form of the knot matrix is 
\begin{equation}
\label{eq:knotmatrix}
{\bf T} = \begin{pmatrix}
-\frac{n + 2\dgdeg}{2} & \frac{-(n + 2\dgdeg)+2}{2} & \cdots & \frac{n-2\dgdeg}{2}\\
-\frac{n + 2\dgdeg-2}{2} & \frac{-(n + 2\dgdeg)+4}{2} & \cdots & \frac{n+2-2\dgdeg}{2}\\
\vdots & \vdots & \ddots & \vdots \\
\frac{2\dgdeg - n }{2} & \frac{2\dgdeg-n +2}{2} & \cdots & \frac{n+2\dgdeg}{2}\\
\end{pmatrix}.
\end{equation}
Each row of the knot matrix gives the B-Spline breaks of the $\gamma^{th}$ B-spline \cite{XLiOne} ($\gamma=1,\dots,{2\dgdeg+1}$).   The $c_\gamma$ are weights of the B-splines, which are determined by ensuring that the kernel satisfies consistency plus $r=2k$ moments.  We further note that 
where $\Kernel{}{H}{\cdot} = \frac{1}{H}\Kernel{}{}{\frac{\cdot}{H}}$ can be viewed as a normalized probability density function.  
We note that Mock and Lax \cite{mock1978} introduced the importance of satisfying moment conditions and pre-processing data.  This allows for recovering accuracy for discontinuous functions -- away from any discontinuities.  The pre-processing of data is important for methods not based on Galerkin orthogonality.  Further, utilizing a linear combination of B-Splines allows for writing derivatives can be written as divided differences of lower order splines,
\[
\frac{\partial^\alpha}{\partial x^\alpha}\bspline{\knots{n}}{x} = \partial_H^\alpha \bspline{\knots{n-\alpha}}{x},
\]
where $\partial_H^\alpha$ represents the $\alpha^{th}$ divided difference.  This ensures that when we pass to the negative order norm, the order of accuracy is not reduced.  These ideas were introduced by 
 Bramble and Schatz \cite{bramble1977higher} and Thome\'{e} \cite{VTH}. 

Here, we note that, using equally spaced knots for the B-spline filter, the Fourier transform of the SIAC kernel is given by 
\begin{equation}
\label{eq:SIACFourier}
\mathcal{F}(K) = \widehat{K}(\xi) = \underbrace{\sinc\left(\frac{\xi}{2}\right)^{n}}_{\text{controls dissipation}}\underbrace{\left(c_{\frac{r+2}{2}}+2\sum_{\gamma = 1}^{\lceil \frac{r}{2}\rceil}\, c_\gamma\cos\left(\left(\gamma-\frac{r+2}{2}\right) \xi\right)\right)}_{\text{moment conditions}} 
\end{equation}
As can be seen, the smoothness chosen for the B-splines controls the amount of dissipation and the number of moments controls the accuracy.  

There are a few choices for extending the filter to multi-dimensions.  One method is via a tensor product:
\[
\Kernel{}{{\bf H}}{\bx} = \Kernel{}{{\dx_1}}{x_1}  \Kernel{}{{\dx_2}}{x_2}\cdots  \Kernel{}{{\dx_d}}{x_d}.
\]
However, for computational efficiency, a rotated one-dimensional filter, the Line SIAC (LSIAC) kernel \cite{LSIAC}, is often used
\begin{equation}
\label{eq:LSIAC}
\Kernel{}{H}{} = \Kernel{}{\Gamma}{} \Rightarrow 
\psijh{j,*}(\overline x,\overline y) = \int_\Gamma
\Kernel{}{\Gamma}{\frac{\Gamma(0) - \Gamma(t)}{\dx_t}}\psijh{j}(\Gamma(t))\, dt.
\end{equation}
In \cite{LSIAC}, results in two- and three-dimensions are demonstrated.  For two-dimensions, filtering is performed along the line $\Gamma(t) = (\overline x, \overline y) +\lambda (\cos(\theta),\ \sin(\theta))$, with an angle of rotation $\theta = \tan\left(\frac{\Delta y}{\Delta x}\right)$.

\subsection{A note on computation \label{sec:SIACcompute}}

An illustration of the improved performance from using the Line SIAC kernel in two-dimensions can be seen by considering a kernel consisting of 5 B-splines of order 3 on a structured mesh.  To filter one point, 196 two-dimensional integrals are required for the implementation of the tensor product filter, while only 21 one-dimensional integrals are required for LSIAC. Computing these integrals require quadratures that respect both B-Spline breaks and mesh breaks.  Computing these breaks is the most costly aspect of the filter.  We note that if polynomials of degree $\dgdeg$ are used in the DG approximation and B-Splines of order $\dgorder$ are used,  $\lceil \frac{2\dgdeg+1}{2}\rceil$ quadrature points per region are required.

\subsection{Post-processing error estimates}

Here we present the main theorem for the $L^2$ error estimate for the postprocessed solution for the time dependent problem. 

\begin{theorem} Let $\psi$ be the exact solution to \eqref{eq:hd_rt_reduced}, and let us assume it satisfies $\psi\in C([0,T];L^2(X;H^s(\sphe)))$. $\psij{j}\in C([0,T];H^{\dgdeg+2}(X))$, $1\leq j\leq \Nang$. if $\psijh{j}$ is the DG solution to \eqref{eqn:upwind_md} with numerical initial conditions $\psijh{j}(\cdot,0)=\Pi^{\dgdeg}\psij{j}_0(\cdot)$, then 
\begin{equation}
    \left(\quadf{\Nang}{\norma{\psi(\cdot,\omega_j)-K_{h}^{2(\dgorder),\dgorder}\star \psijh{j}}_{L^2(\Omega_0)}}\right)^{1/2} \leq C_{\mathrm{T}}h^{2\dgdeg+1/2}+D\Nang^{-s},
\end{equation}
where $C_{\mathrm{T}}$ depends on the upper bound of $\left(\quadf{\Nang}{\norma{\partial_h^{\lambda} \psij{j}}_{\dgdeg+2,X}}^2\right)^{1/2}$, for all $|\lambda|\leq \dgorder$ and the constant $C_{\mathrm{P}}$ in Theorem \ref{thm:bramble_schatz} and $D$ depends on upper bounds of $\norma{\psij{}}_{L^2(X;H^s(\sphe))}$.
\end{theorem}
\begin{proof}
    A direct application of triangle inequality gives,
    \begin{align}
    \left(\quadf{\Nang}{\norma{\psi(\cdot,\omega_j)-K_{h}^{2(\dgorder),\dgorder}\star \psijh{j}}_{L^2(\Omega_0)}}\right)^{1/2}  &\leq  \left(\quadf{\Nang}{\norma{\psi(\cdot,\omega_j)-K_{h}^{2(\dgorder),\dgorder}\star \psij{j}}_{L^2(\Omega_0)}}\right)^{1/2}  \notag\\
         &+\left(\quadf{\Nang}{\norma{\psij{j}(\cdot,\omega_j)-K_{h}^{2(\dgorder),\dgorder}\star \psijh{j}}_{L^2(\Omega_0)}}\right)^{1/2} \notag
    \end{align}
    The first term in the inequality comes from the proof of Theorem \ref{thm:total_error_bound}. To bound the second term, we use the fact that since the scattering $\sigma$ is constant, then for all multi-indices $\lambda$, $\partial_h^{\lambda}\psi$ satisfies equation \eqref{eq:hd_rt_reduced}. Then by Theorem \ref{thm:negative_order_estimate} 
    \begin{equation}
        \norma{\partial_h^{\lambda}(\psi-\psi_h)}_{-(\dgorder),\BX}\leq C_{\lambda}h^{2\dgdeg+1/2},
    \end{equation}
    with $C_{\lambda}$ depending on upper bounds of $\norma{\partial_h^{\lambda}\psi^j}_{\dgdeg+2,\BX}$, $1\leq j\leq\Nang$. Then the conclusion follows applying Theorem \ref{thm:bramble_schatz}.
\end{proof}
\section{Numerical results\label{sec:num}}
Here, we demonstrate the performance of SIAC filter through a series of numerical examples.
Due to its superior efficiency discussed in Sec \ref{sec:SIACcompute}, we apply the line SIAC filter \cite{LSIAC} in all our numerical tests.

The linear system resulting from the discretization of the steady-state problem or implicit time-marching is solved through Source Iteration with Diffusion Synthetic Acceleration (SI-DSA) \cite{adjerid2002posteriori}. A partially consistent DSA strategy is applied as our preconditioner (see \cite{wareing1993new} and Appendix A of \cite{peng2025flexible} for details). The diffusion equation inside our DSA preconditioner is solved by conjugate gradient method with algebraic multigrid (AMG) preconditioner.  We set the stopping criteria of the source iteration as $||\barp^{(\ell)}-\barp^{(\ell-1)}||<\epsilon_{\textrm{SI}}$  with $\barp^{(\ell)}$ be the density in the $\ell$-th iteration. We set $\epsilon_{\textrm{SI}}$ as $10^{-10}$ for $K\leq 2$ and $10^{-11}$ for $K=3$.

We implement our code in the {\tt{Julia}} language, leveraging {\tt{IterativeSolvers.jl}} package for Krylov solver, {\tt{AlgebraMultigrid.jl}} package for AMG preconditioner and {\tt{MSIAC.jl}} \cite{docampo2023magic} for the SIAC filter. Numerical tests are performed on a MacBook with Apple M1 chip. 
\subsection{Accuracy test}
We present a series of accuracy tests to demonstrate superconvergence and the computational saving gained to reach the same level of accuracy by applying the SIAC filter.  
\subsubsection{Steady-state problem\label{sec:num-accuracy-steady}}
The first problem that we consider the computational domain $[-1,1]^2$  with vacuum boundary conditions. A uniform mesh with $\Nx\times N_y$ rectangular elements is applied to partition the computational domain. The mesh size is defined to be $\dx=\min(1/\Nx,1/N_y)$. We impose source terms so that the manufactured solution $f(x,y,v_x,v_y)=\sin(\pi x)\sin(\pi y)$ is satisfied. We consider two different material properties: (1) constant material, $\sigma_s(x,y)=1$ and $\sigma_a(x,y)=0$; and (2) variable material, $\sigma_s(x,y)=2+\sin(16\pi x)\sin(16\pi y)$ and $\sigma_a=0$.

\textbf{Superconvergence.} In Fig. \ref{fig:steady_state_accuracy_test}, we present $L^2$ error of the numerical solution before and after post-processing with the SIAC filter. Before applying the filter, we observe the expected $(\dgorder)$-th order of accuracy. After applying the SIAC filter, we observe approximately $3.5$ order of accuracy for $Q_1$ elements, and $2(\dgdeg+2)$-th order of accuracy for $\dgdeg=2,3$. This observation matches the theories given in Sec. \ref{sec:superconvergence-filter}. 

We also observe that the filtered solution is more accurate  when the mesh resolution is sufficiently high, while it may not be more accurate on a coarse mesh.

\textbf{Efficiency gain.} 
In this example, directly applying the SIAC filter to $\barp$ always leads to less than $10\%$ more computational time.
In Fig. \ref{fig:steady_state_accuracy_test_speed}, we present the relation between the $L^2$ error and the computational time. We observe that the SIAC filter enables us to obtain significantly more accurate results with almost the same computational time as when a more refined mesh is used. Moreover, we observe that, after a break-even point, applying the SIAC filter takes less computational time to reach the desired level accuracy compared to $h$-refinement.

For constant scattering, to reach an error of pproximately $7\times 10^{-5}$ with $Q_1$ elements, it takes approximately $15.42$ seconds with $\dx=N_y=20$ when the SIAC filter is applied. It takes approximately $140.06$ seconds with $\Nx=N_y=80$ without the help of SIAC filter to reach the same level of accuracy. In this case, SIAC filter leads to approximately $9$ times acceleration compared to computing on a refined mesh. 

For the constant scattering, when $Q_2$ element are used, the SIAC filters enables us to obtain $6.63\times 10^{-8}$ in the $L^2$ error with approximately $160.71$ seconds and $\Nx=N_y=40$, while only $1.35\times10^{-7}$ for the $L^2$ error is achieved with $621.45$ seconds and $\Nx=N_y=80$ without filtering. In this case, applying SIAC filter permits us to obtain $2$ times more accurate results with only $25\%$ computational time compared to refining the mesh.

Similar observations can be seen for the variable scattering case as well.

\begin{figure}[htbp]
    \centering
    \includegraphics[width=0.45\textwidth]{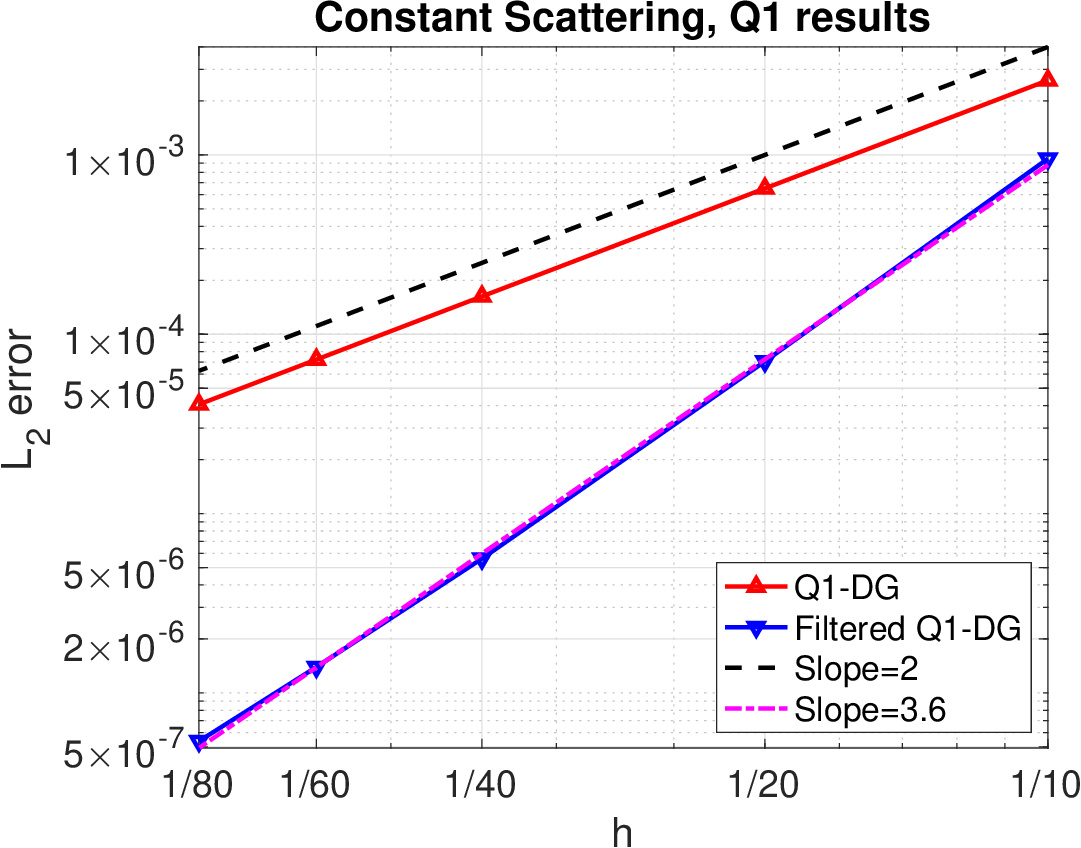}
    \includegraphics[width=0.45\textwidth]{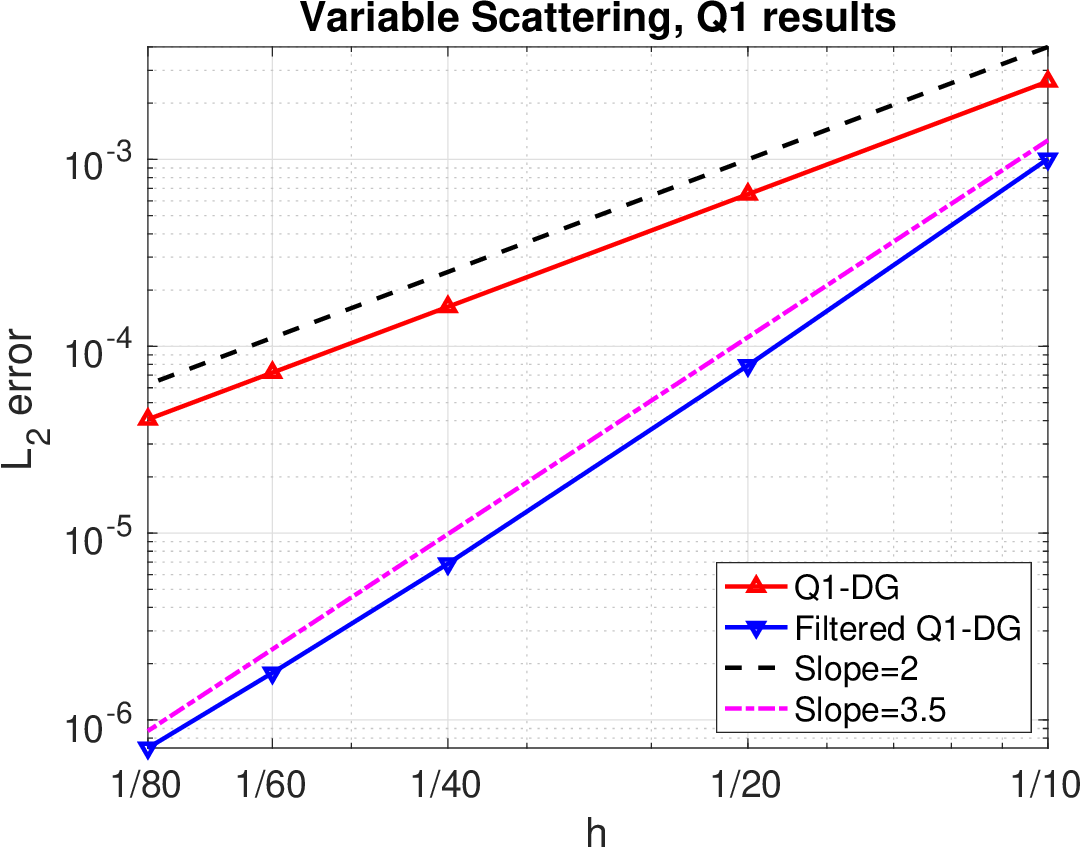}
    \includegraphics[width=0.45\textwidth]{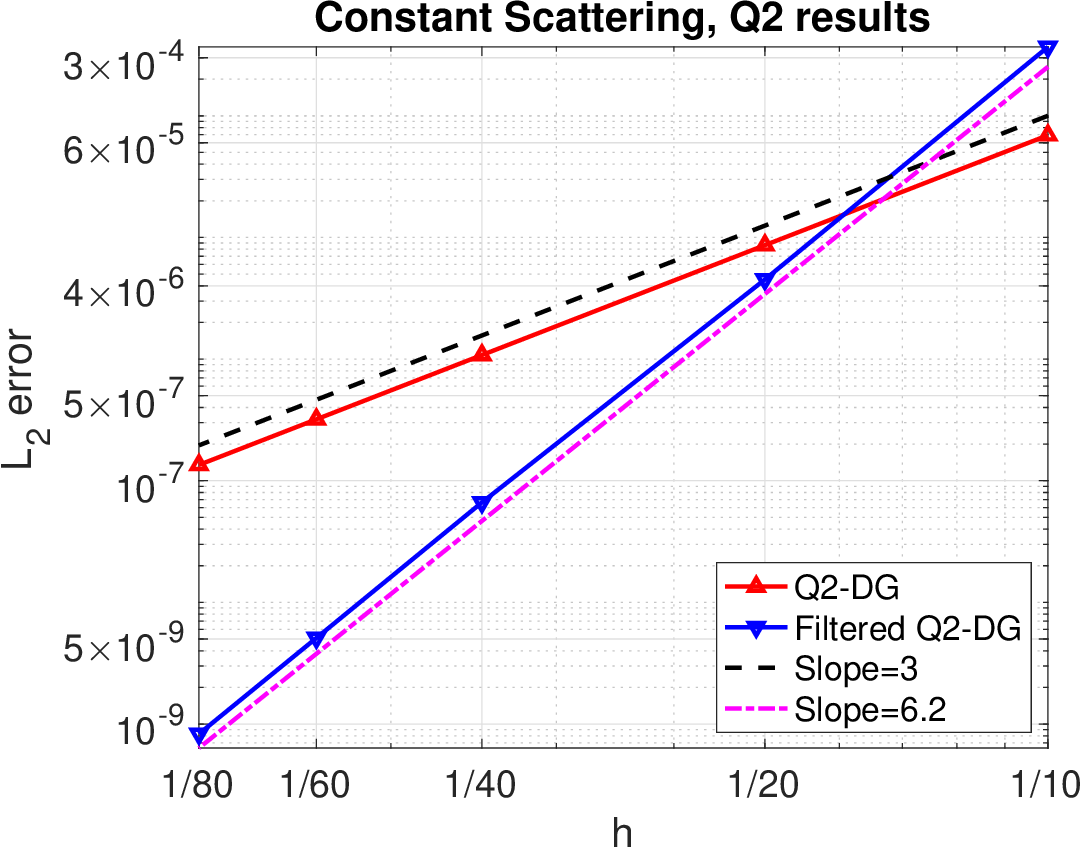}
    \includegraphics[width=0.45\textwidth]{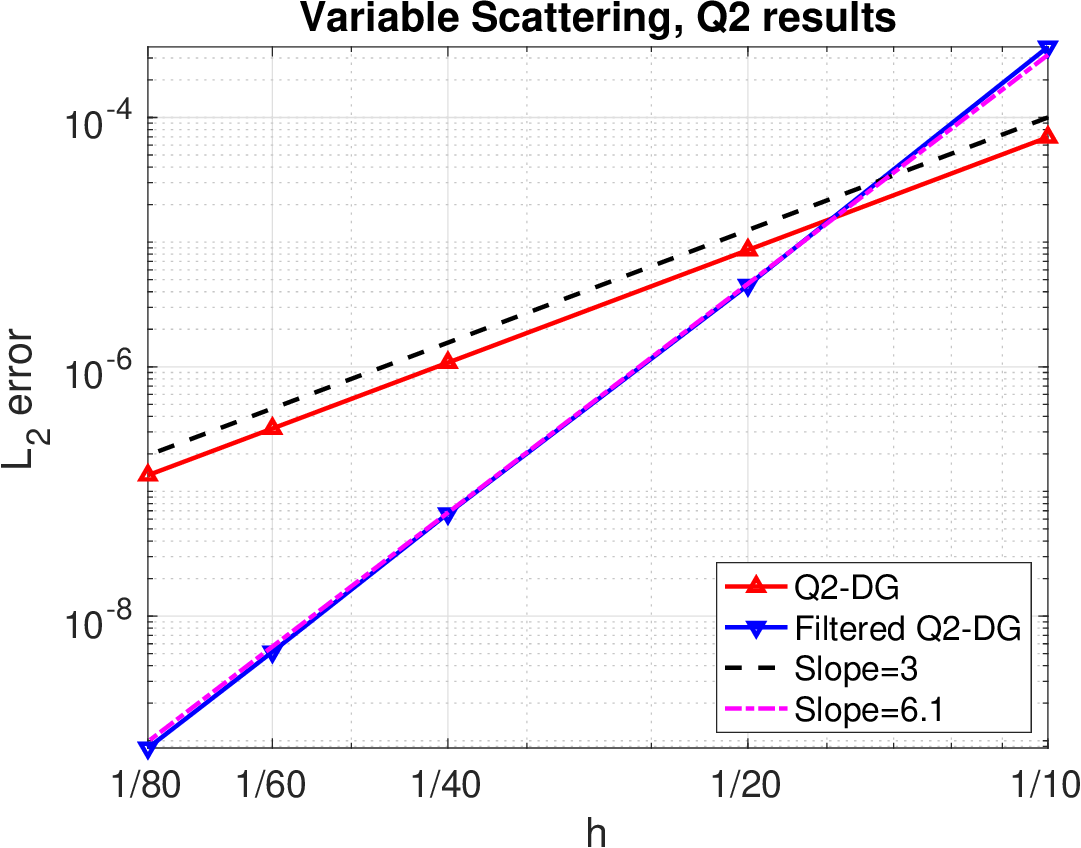}
    \includegraphics[width=0.45\textwidth]{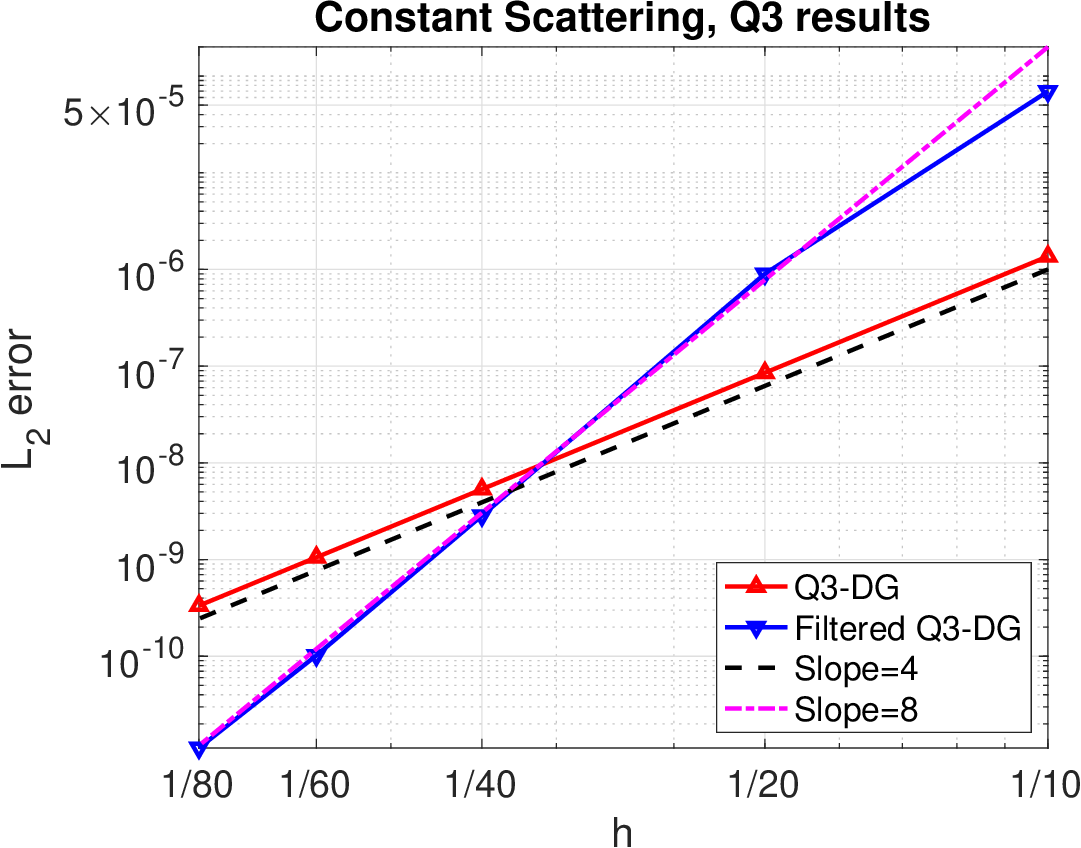}
    \includegraphics[width=0.45\textwidth]{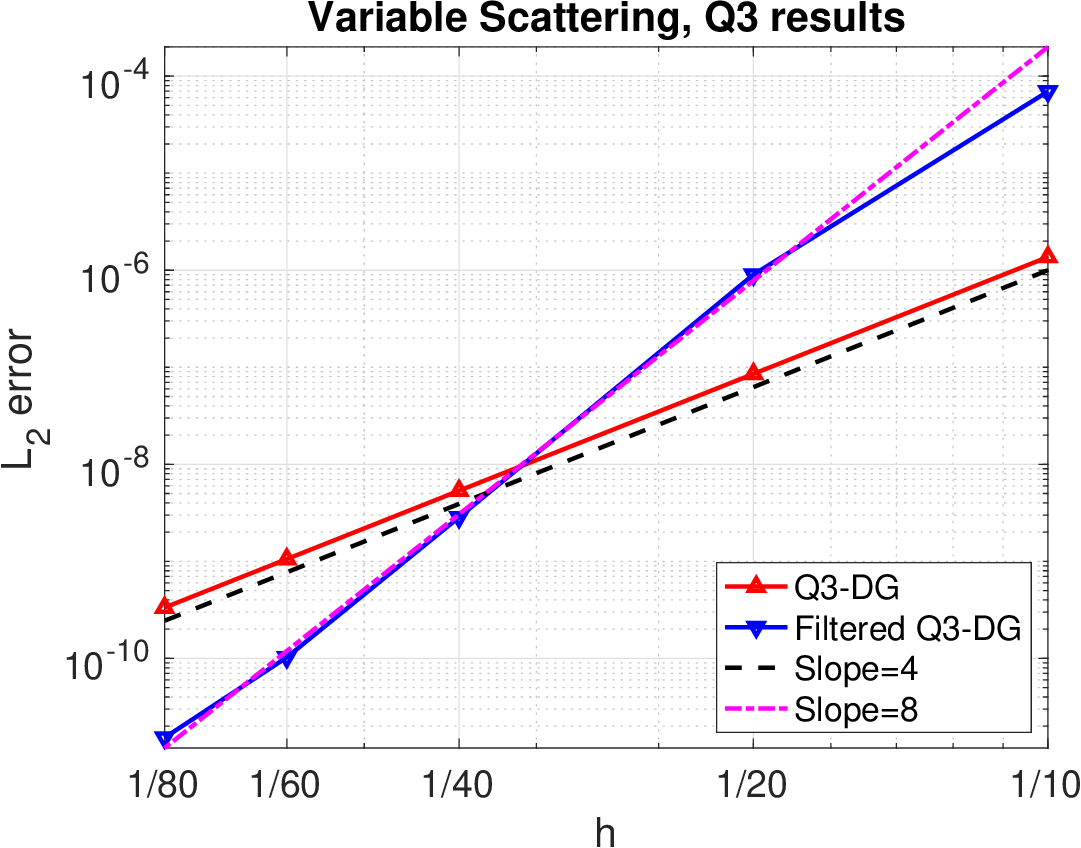}
    \caption{$L^2$ error vs mesh resolution for the steady-state accuracy test in Section \ref{sec:num-accuracy-steady}. Left: constant scattering. Right: variable scattering.\label{fig:steady_state_accuracy_test}} 
\end{figure}

\begin{figure}[htbp]
    \centering
    \includegraphics[width=0.45\textwidth]{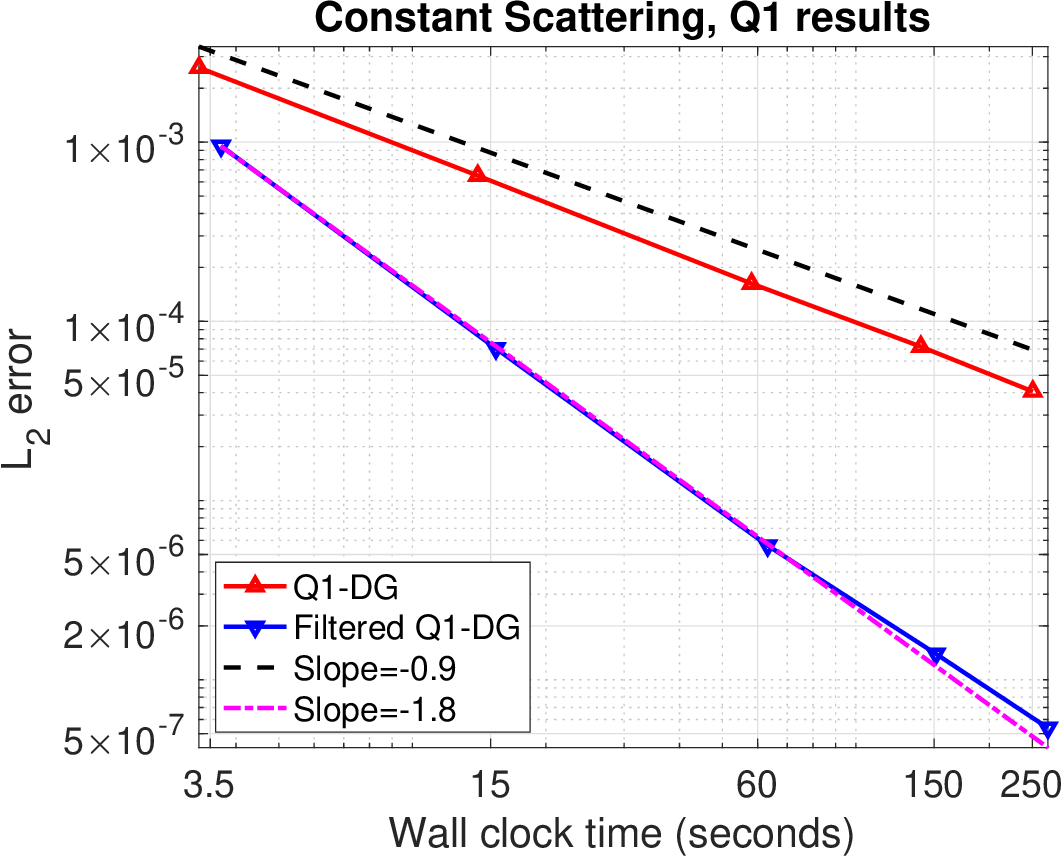}
    \includegraphics[width=0.45\textwidth]{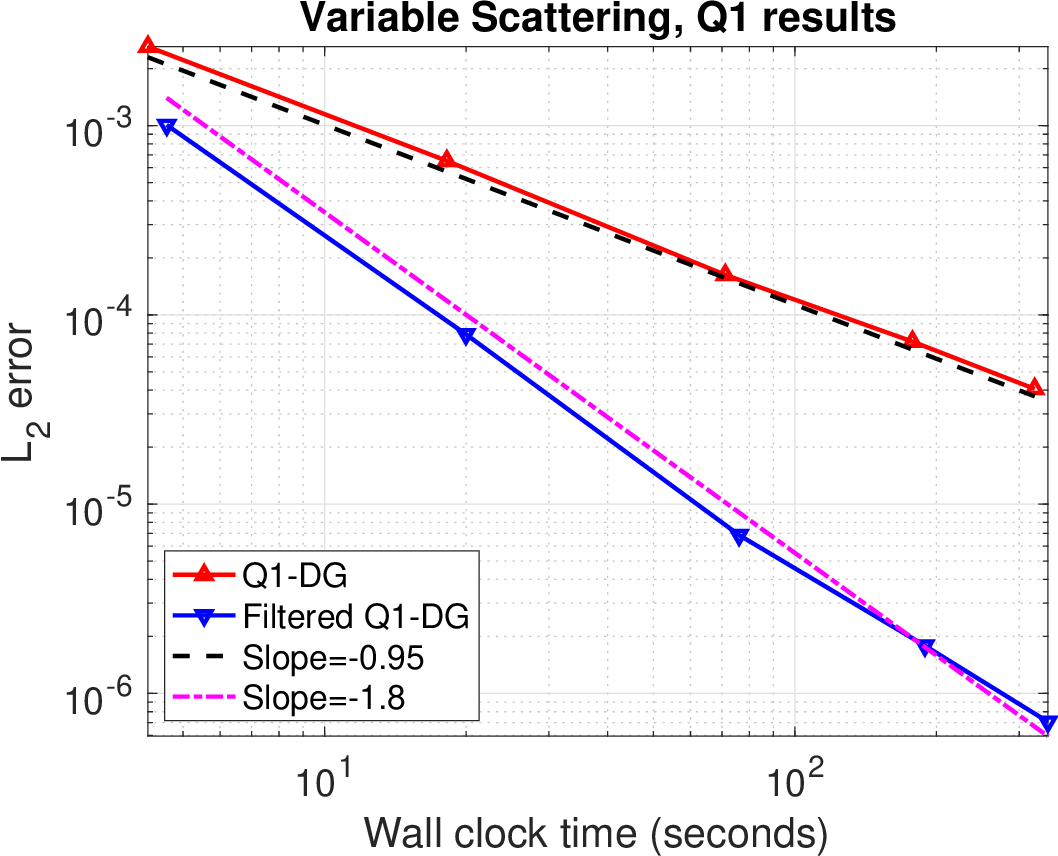}
    \includegraphics[width=0.45\textwidth]{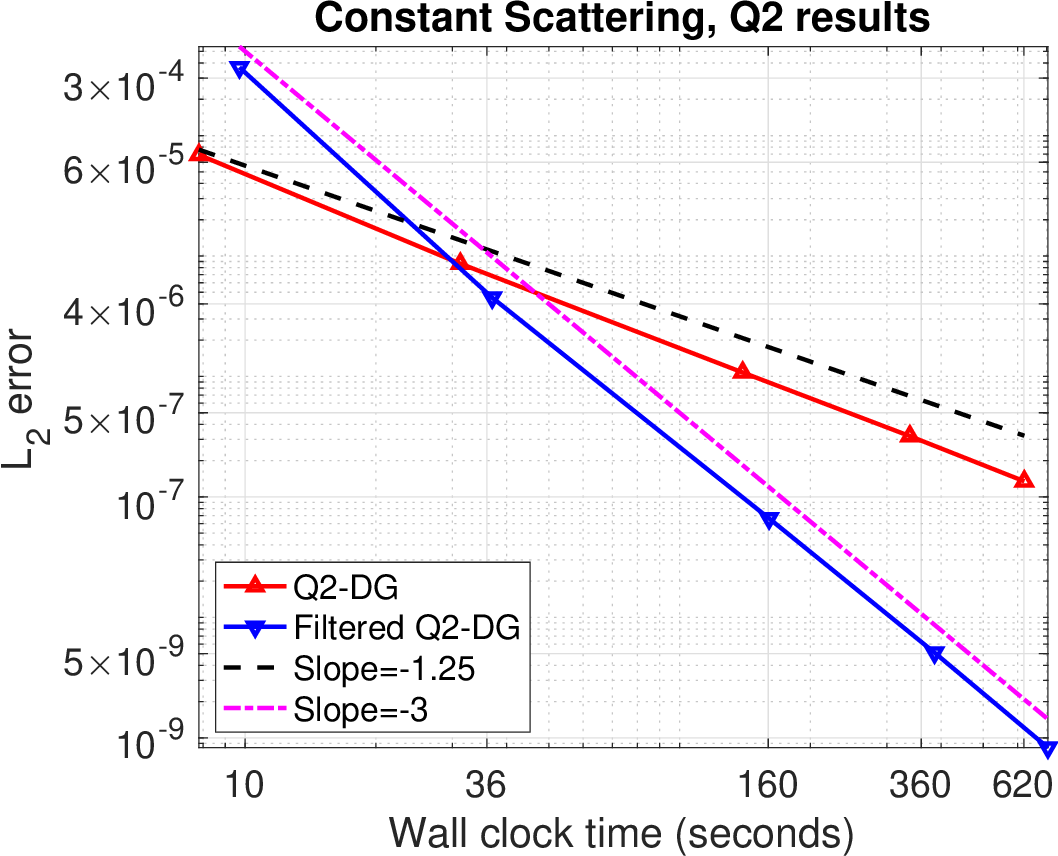}
    \includegraphics[width=0.45\textwidth]{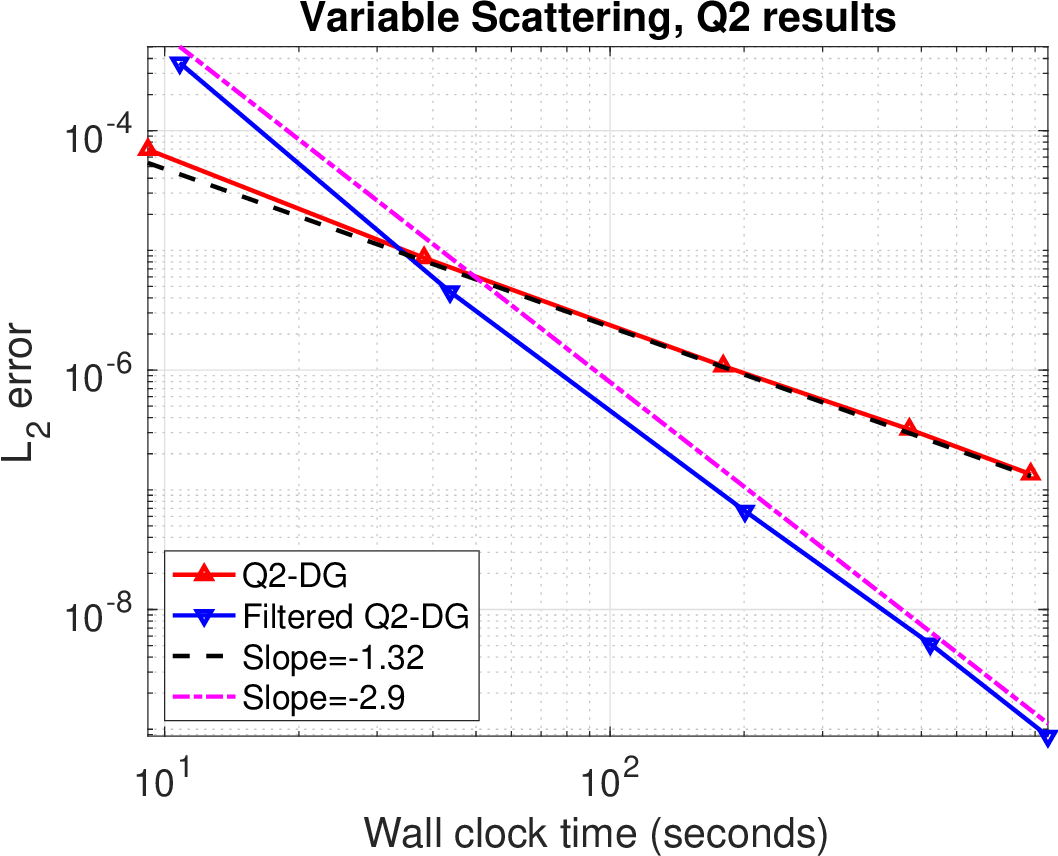}
    \includegraphics[width=0.45\textwidth]{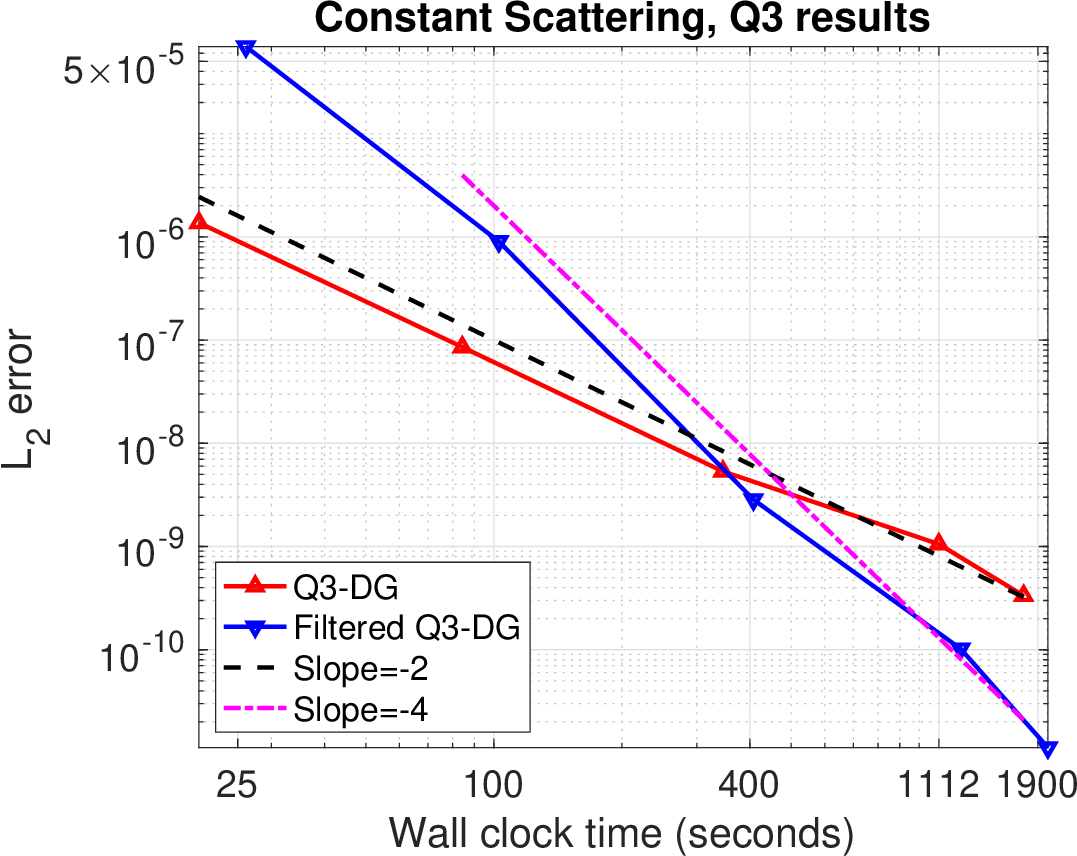}
    \includegraphics[width=0.45\textwidth]{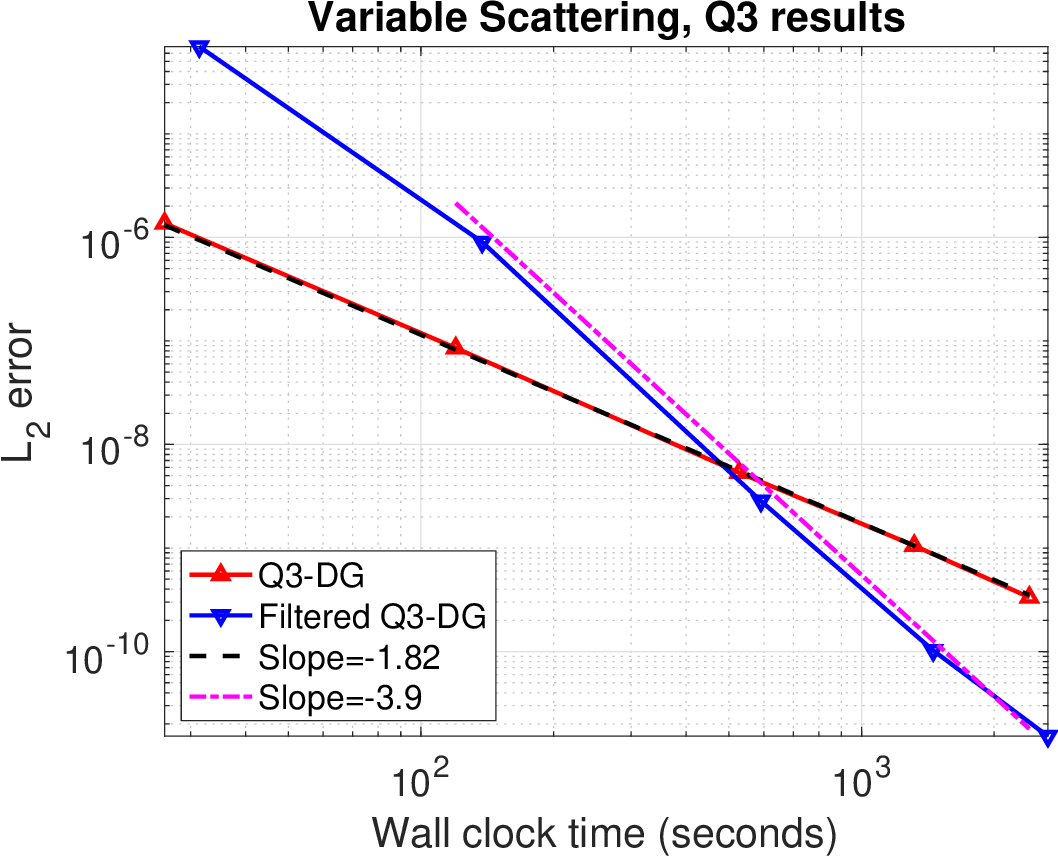}
    \caption{$L^2$ error vs computational time for the steady-state accuracy test in Section \ref{sec:num-accuracy-steady}. Left: constant scattering. Right: variable scattering.\label{fig:steady_state_accuracy_test_speed}} 
\end{figure}

\subsubsection{Time dependent problem\label{sec:num-accuracy-td}}
For the time-dependent case, we consider the computational domain $[-1,1]^2$ with vaccum boundary conditions. A uniform mesh with $N_x\times N_y$ rectangular elements is used with mesh size $h=\min(1/N_x,1/N_y)$. The material properties used in this example are $\sigs(x,y)=1$ and $\siga(x,y)=0$. We impose source terms so that $f(t,x,y,v_x,v_y)=\exp(-t)\sin(\pi x)\sin(\pi y)$ is an exact solution. We run the simulation from time $t=0$ to $t=0.5$.

\textbf{Superconvergence.}
For the time-dependent problem, we test both a BDF$2$ and BDF$3$ schemes. For these methods, the temporal error will become dominant when the SIAC filter is applied and, in order to observe spatial superconvergence, either a higher-order temporal discretization should be used, or the CFL should be modified accordingly. We consider the following discretizations (1) $Q1$-BDF$2$ with a time step size $\dt= h$, (2) $Q1$-BDF$3$ with a time step size $\dt=h$ and (3) $Q2$-BDF$3$ with $\dt=4h^{\frac{5}{3}}$. 

We present the relation between the $L^2-$error and the mesh size in Fig. \ref{fig:time_dependent_accuracy_test}. From the top left plot, we observe that SIAC filter is able to improve the accuracy of $Q_1$-BDF$2$, but the accuracy order is still second order due to the dominance of the temporal error. From the top right result, we observe that the SIAC filter improves the accuracy order of $Q_1$-BDF$3$ from second order to third order. From the bottom left figure, slightly higher than $5$-th order accuracy is observed for $Q_2$-BDF$3$ with $\dt=4h^{\frac{5}{3}}$ when SIAC filter is applied. In summary, the SIAC filter is able to improve the order of accuracy to $2\dgorder$ with sufficiently accurate temporal discretization.

\textbf{Efficiency gain.}  We compare the efficiency of $Q_1$-BDF$2$, $Q_1$-BDF$3$ without post-processing with the filtered $Q_1$-BDF$3$. The time step size is chosen as $\dt=h$ for all methods. 

Though BDF$3$ involves  more vector operations per time step than BDF$2$, both of them only require one linear solve per time step through SI-DSA. Since the linear solve SI-DSA takes significantly longer time than additional vector operations in BDF$3$, the overall computational efficiency of these two time integrators are comparable. 

In the bottom right plot of Fig. \ref{fig:time_dependent_accuracy_test}, we present the relation between the computational time, namely $T_{\textrm{comp.}}$, and the $L^2-$error. We observe that the $L^2-$error scales roughly as $T_{\textrm{comp.}}^{-0.8}$ for $Q_1$-BDF$2$ and $Q_1$-BDF$3$ without filtering, while it scales as $T_{\textrm{comp.}}^{-1.3}$ for $Q_1$-BDF$3$ post-processed using the SIAC filter. Utilizing the same computational time, $Q_1$-BDF$3$ with post-processing achieves significantly more accurate results.

Specifically, with $N_x=N_y=96$, $Q_1$-BDF$2$ and $Q_1$-BDF$3$ before filtering achieves $2.52\times10^{-5}$ and $1.71\times 10^{-5}$ $L^2-$error with $3004.53$ seconds and approximately $3049.48$ seconds computational time, respectively. By applying SIAC filter, $Q_1$-BDF$3$ with $N_x=N_y=32$ is able to achieve $7.70\times10^{-6}$ error with only $152.94$ seconds. Consequently, using the SIAC filter and BDF$3$ achieves an approximately $3.26$ times more accurate result with $19.65$ times acceleration over $Q_1$-BDF$2$, and a $2.22$ times more accurate result with  $19.94$ times acceleration over $Q_2$-BDF$3$.

Compared to the steady-state case, the time of post-processing using the SIAC filter is less significant -- it is smaller than $5\%$ -- since the total number of source iterations is significantly higher in the time-marching case.  

\begin{figure}[htbp]
    \centering
    \includegraphics[width=0.45\textwidth]{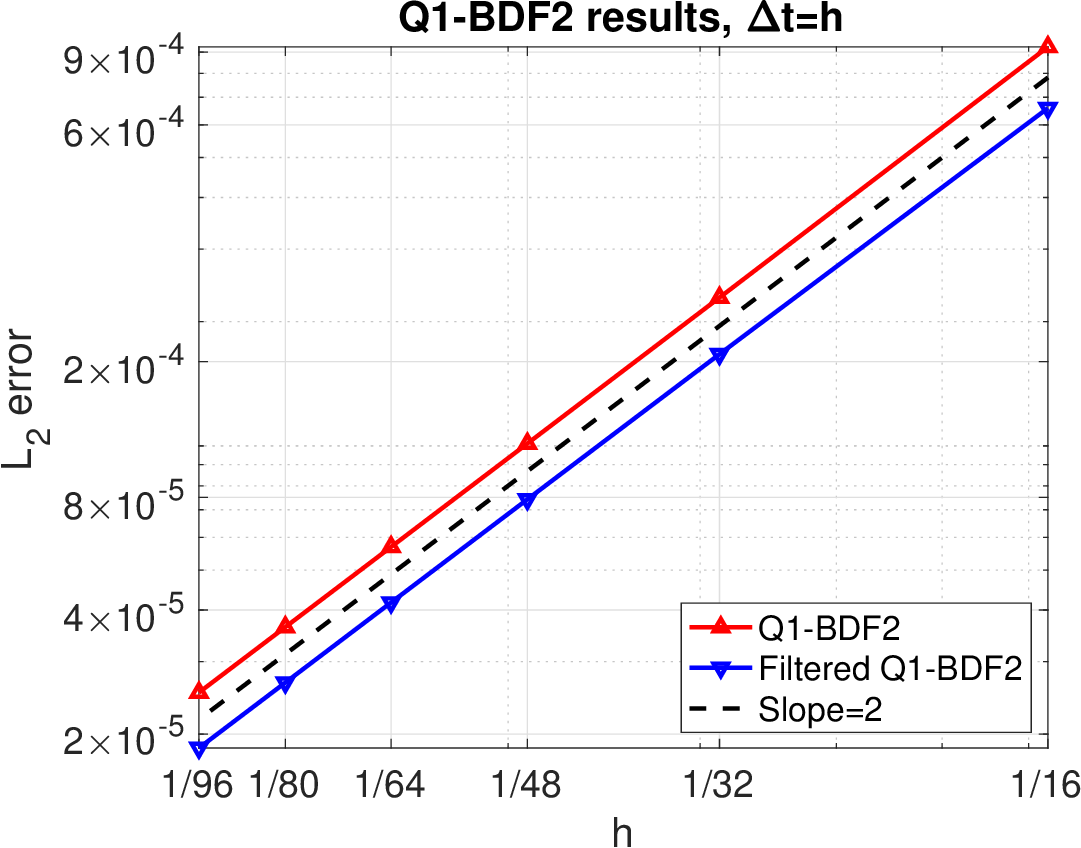}
    \includegraphics[width=0.45\textwidth]{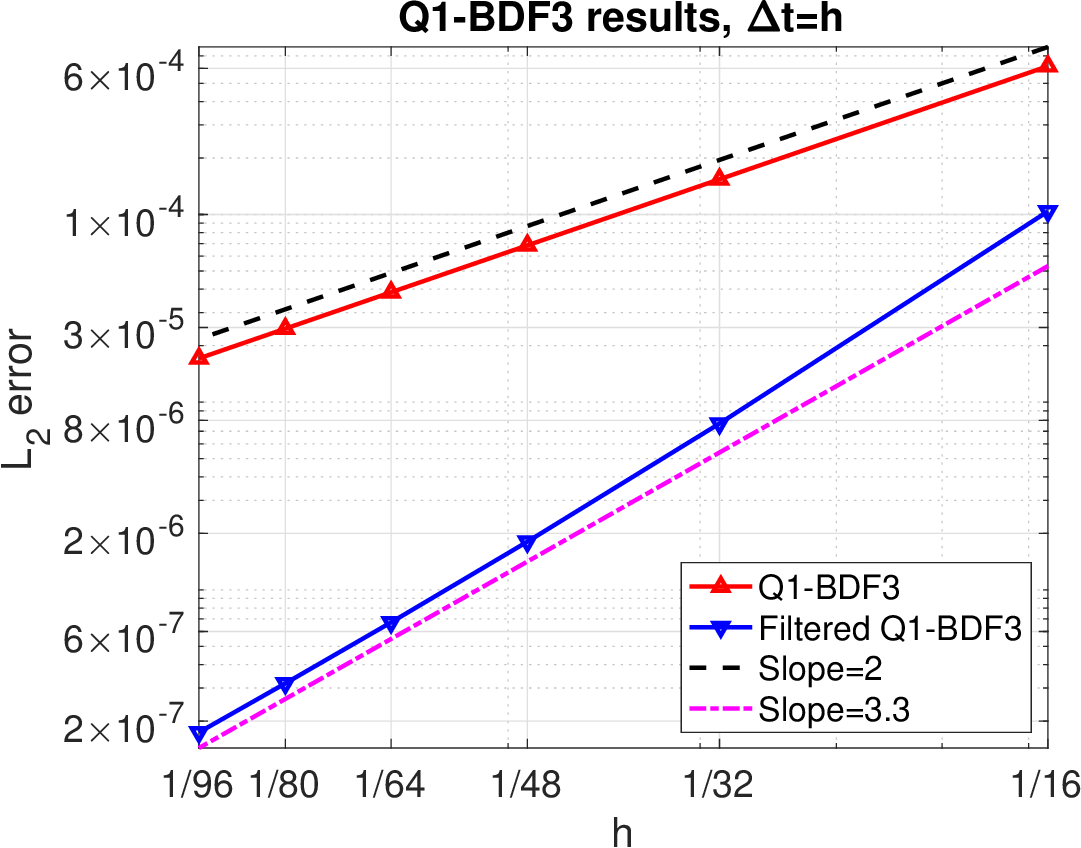}
    \includegraphics[width=0.45\textwidth]{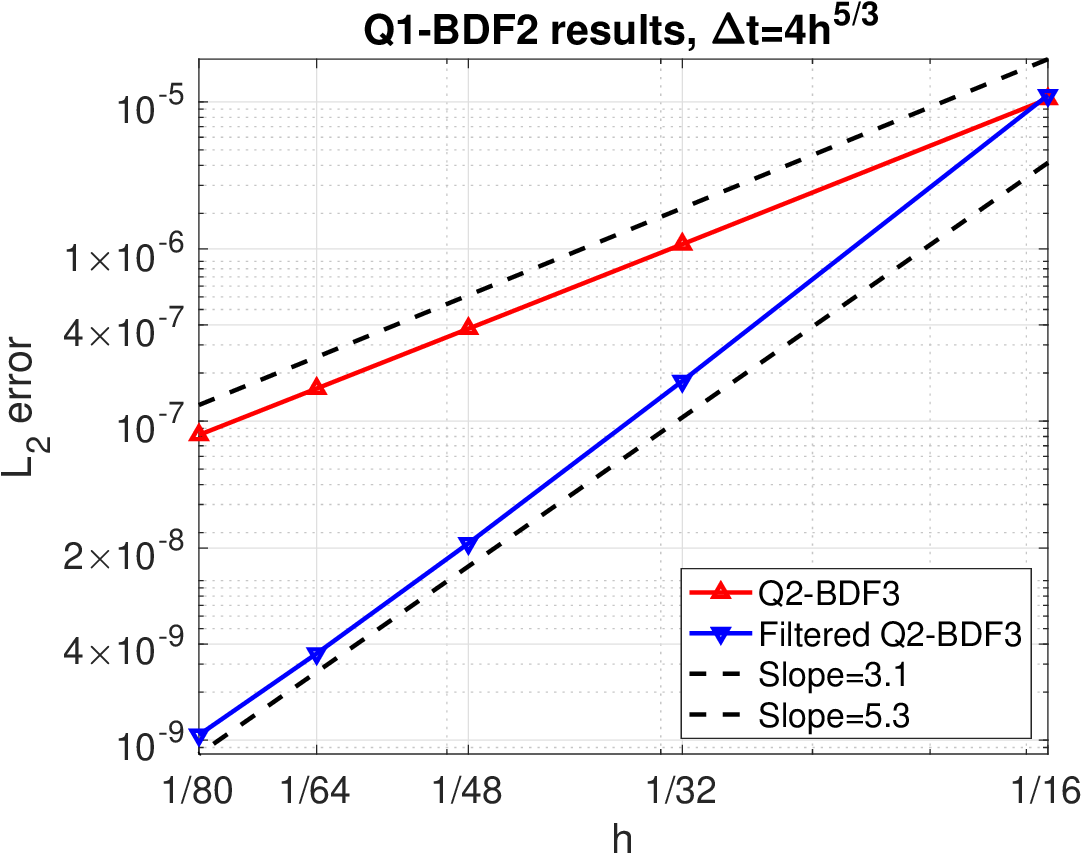}
    \includegraphics[width=0.45\textwidth]{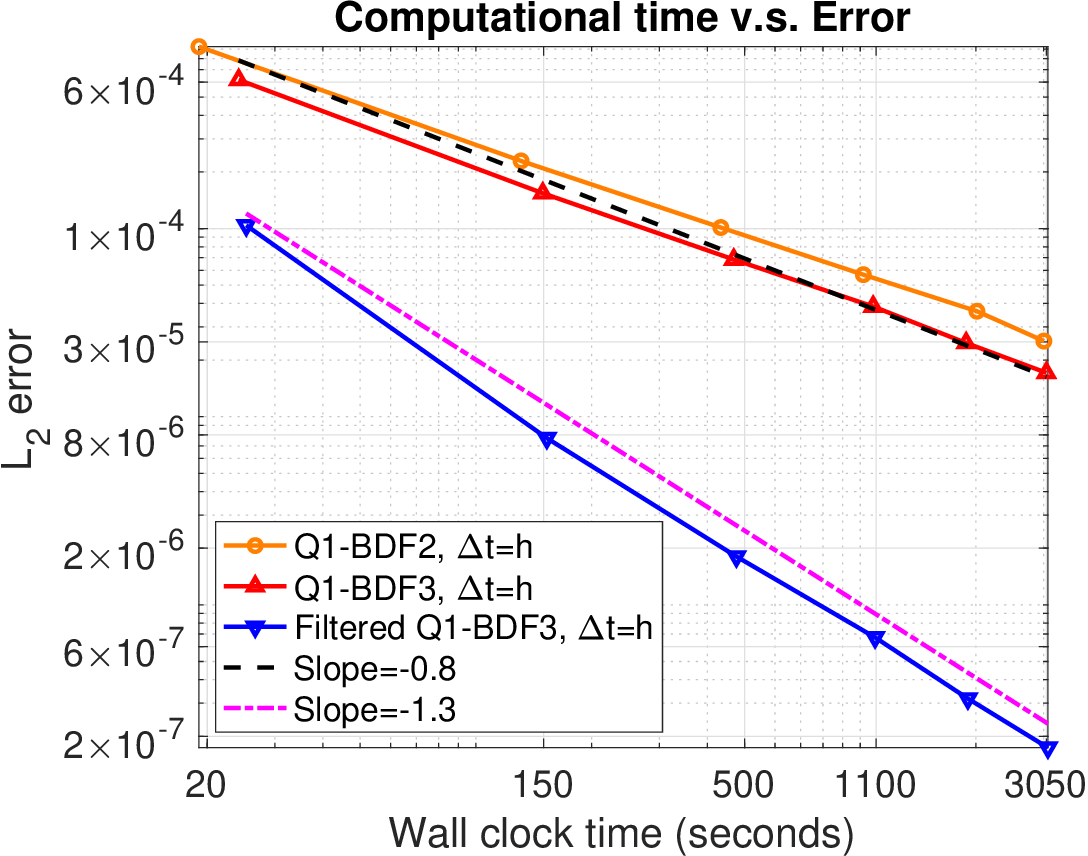}
    \caption{Results for the time-dependent accuracy test in Section \ref{sec:num-accuracy-td}. Top left: $L^2$ error v.s. $\dx$ for $Q_1$-BDF$2$ with $\dt=\dx$. Top right: $L^2$ error v.s. $h$ for $Q_1$-BDF$3$ with $\dt=h$. Bottom left: $L^2$ error v.s. $\dx$ for $Q_2$-BDF$3$ with $\dt=\dx^{\frac{5}{3}}$. \label{fig:time_dependent_accuracy_test}} 
\end{figure}

\subsection{Steady-state with variable scattering\label{sec:variable-scattering}}
We consider a steady-state problem on the computational domain $[-1,1]^2$ with vacuum boundary conditions and a Gaussian source $G(x,y)=\frac{10}{\pi}\exp(-100(x^2+y^2))$. There is no absorption, while the scattering cross section is defined as 
\begin{equation}
   \sigs(x,y) = \left\{
\begin{array}{ll}
99r^{4}(2-r^4)^2 + 1, \quad & r = \sqrt{x^2+y^2}\leq 1, \\
100, & \textrm{otherwise}.
\end{array}
\right.
\end{equation}
The configuration of $\sigs$ and a reference solution generated with $Q_1$ DG using $128\times128$ rectangular mesh in space and CL($80,40$) quadrature rule is presented in Fig. \ref{fig:variable-scattering}.

As discussed in \cite{li2024random,fu2025random}, the regularity of RTE is low in the angular space, and the $S_N$ method may suffer from significant order reduction in the angular space. With low regularity in the angular space, we are not able to observe superconvergence when applying the SIAC filter due to the dominating angular error. However, as shown in Fig. \ref{fig:variable-scattering}, the SIAC filter still removes non-physical oscillations in the center of computational domain and improves the resolution of the solution on a coarse space-angle mesh with $32\times32$ elements in the physical space and CL($20,10$).

\begin{figure}[htbp]
    \centering
    \includegraphics[height=0.4\textwidth,trim=10mm 0mm 25mm 0mm, clip]{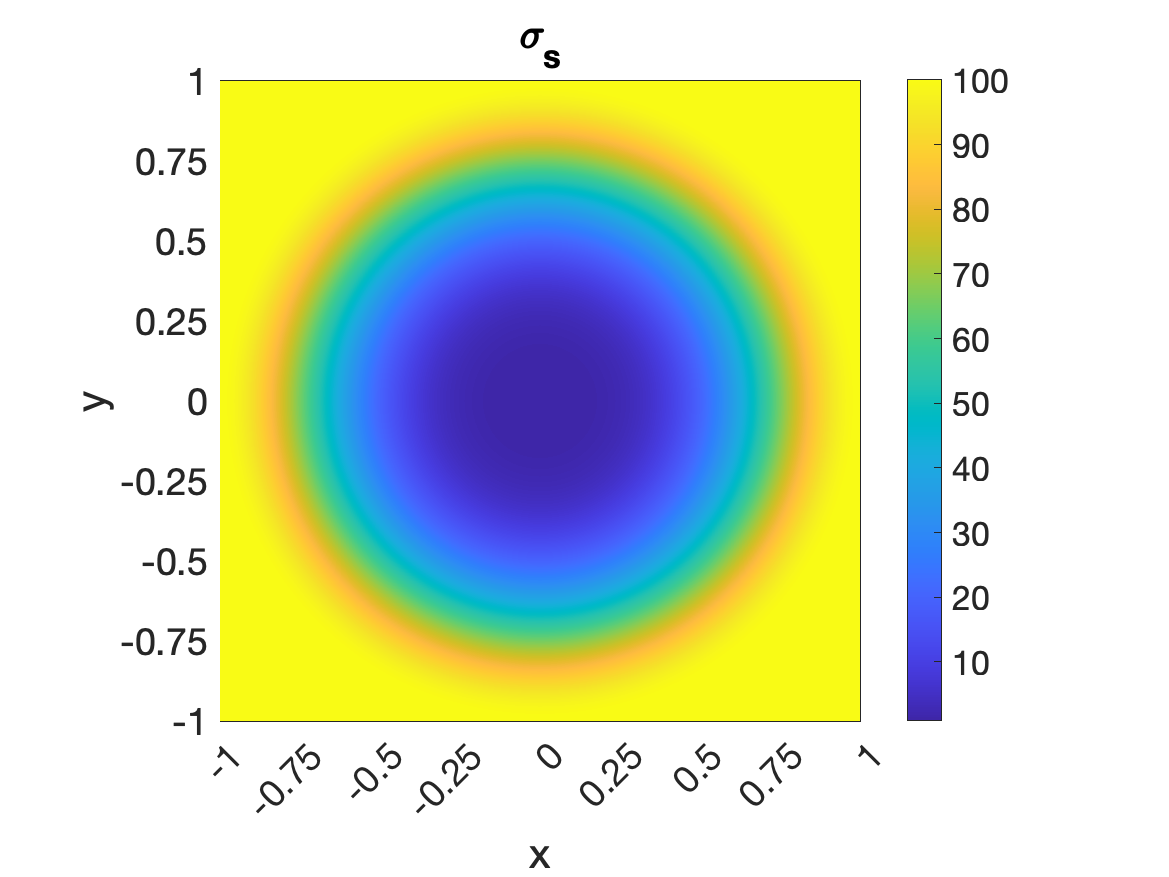}
    \includegraphics[width=0.4\textwidth]{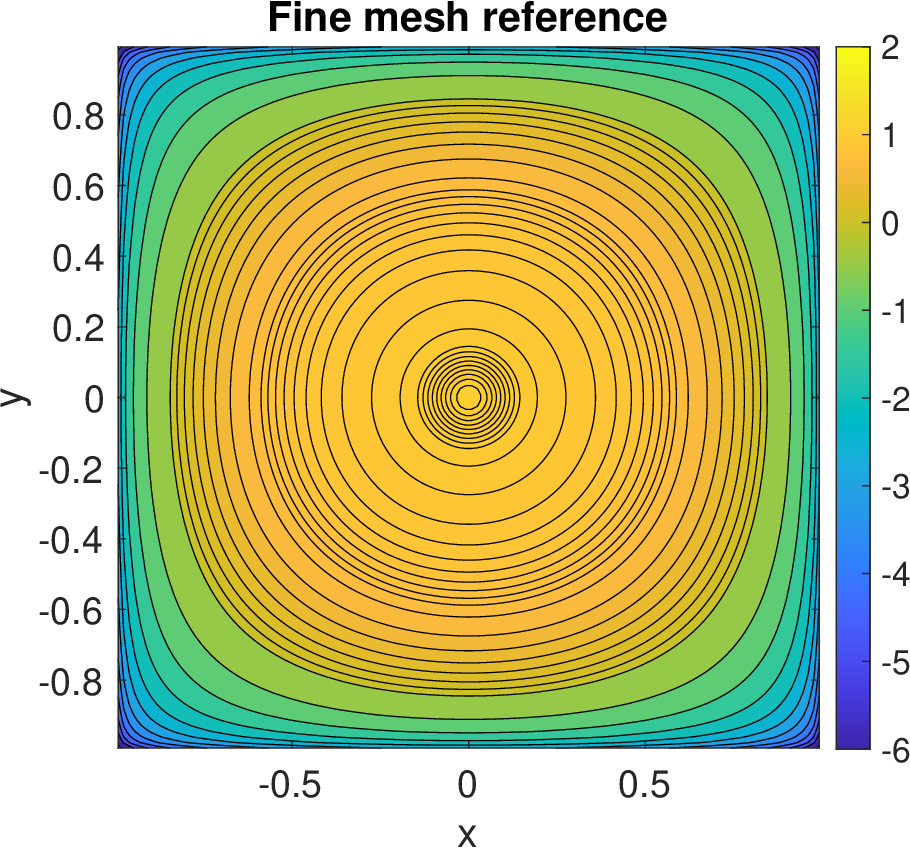}
    \includegraphics[width=0.4\textwidth]{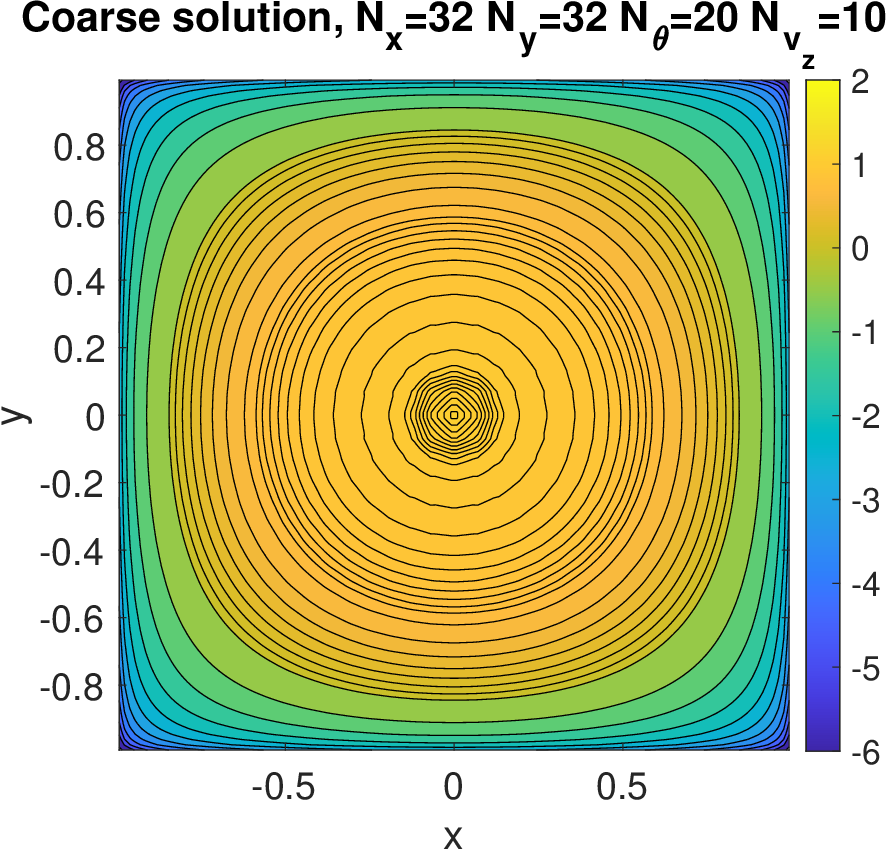}
    \includegraphics[width=0.4\textwidth]{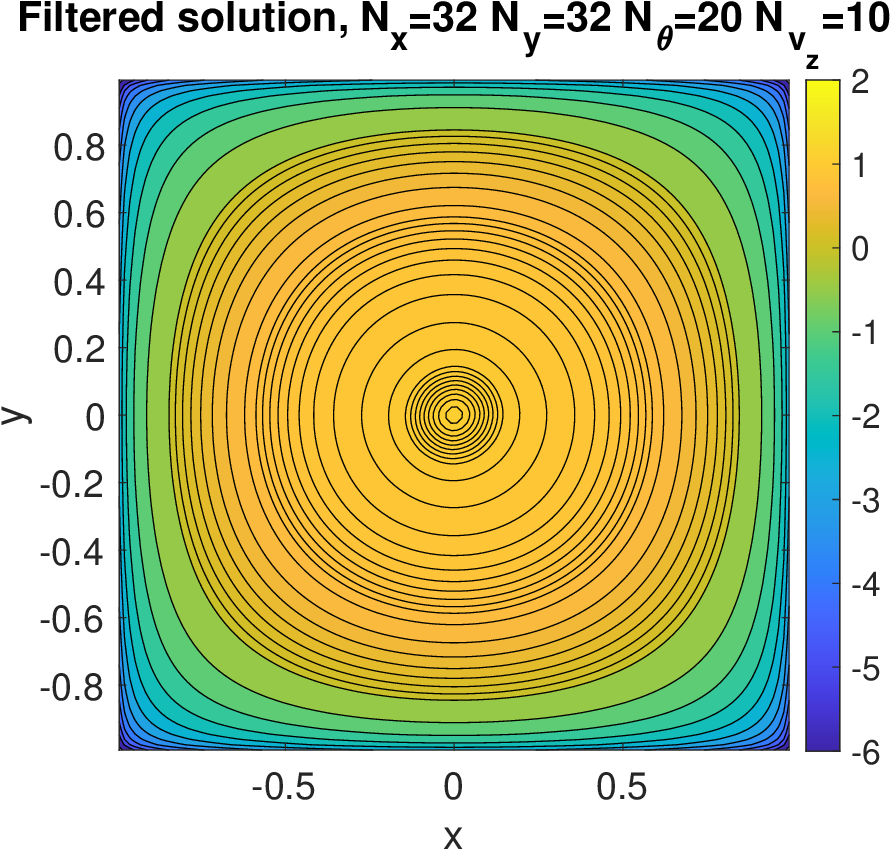}

    \caption{ Results for the multiscale variable scattering problem in Sec. \ref{sec:variable-scattering}. Solutions are presented under log-scale. Top left: configuration of the scattering cross section. Top right: reference solution on a fine mesh. Bottom left: solution on a coarse mesh. Bottom right: filtered solution on the coarse mesh.\label{fig:variable-scattering}} 
\end{figure}

\section{Conclusions\label{sec:con}}
In this paper, we theoretically and computationally present the benefit of applying SIAC filters to  the upwind DG method for solving the steady-state and time-dependent RTE. 
\begin{enumerate}
    \item We have proven $(2\dgdeg+2)$-th order accuracy for the steady state RTE at the outflow edge and $(\dgdeg+2)$-th order accuracy at the interior roots of the Radau points on each spatial element, and $(2\dgdeg+\frac{1}{2})$-th order accuracy for the time-dependent problem with respect to a weighted negative-order norm. Our numerical results validate $(2\dgdeg+2)$-th order superconvergence for steady state problem and demonstrate $(2\dgorder)$-th order superconvergence for the time-dependent problem. 
    
    \item Furthermore, we numerically demonstrate that directly applying SIAC filters to low-dimensional macroscopic density can greatly reduce the computational time to reach a desired level of accuracy. 
\end{enumerate}

Potential future directions are as follows: (1) For varying cross sections, our current superconvergence proof can be extended by establishing divided difference results for the DG approximation. Though numerically observed superconvergnce, a complete superconvergence proof for spatial cross sections is still an open questions. Additionally, analysis for problems  involving nonlinear thermal radiation, multi-energy groups and anisotropic scattering is challenging and worth investigation.
(2) Our current analysis mainly focuses on improving convergence order in physical space.  However, another important area of exploration is enhancing angular accuracy and mitigating ray effects through post-processing in angular space. 
\section{Declarations}
\subsection*{Acknowledgments}
A. Galindo-Olarte gratefully acknowledges support from the Oden Institute for Computational Engineering and Sciences.  Z. Peng also acknowledges support from the Hong Kong Research Grants Council (grant ECS-26302724 and GRF-16306825). The work of J. Ryan was funded by the Swedish Research Council (VR grant 2022-03528). Z. Peng and J. Ryan were also funded by internal support from the HKUST-KTH Global Knowledge Network Awards. 
\bibliographystyle{plain}
\bibliography{refer.bib}

\appendix
\label{sec:appendix}
\section{Analysis of spatial error, \cref{thm:dg_error_time_dep}}

\begin{proof}
    Summing \eqref{eqn:upwind_md} over all $K\in\Th$ we have that $\psi_h$ satisifies
    
    \begin{equation}\int_{Th}\partial_t\psi_h^j\tau_h\,d\bx+B_h(\psi_h^j,\tau_h;\angj{j})=\sigma\int_{\Th}(\overline{\Psi_h}-\psij{j}_h)\tau_h\,d\bx
    \end{equation}

    where
    \begin{equation}
        B_h(\psi_h^j,\tau_h;\angj{j})=-\int_{\Th}\psi_h^j(\angj{j}\cdot\nabla_x\tau_h)\,d\bx+\int_{\Eh}\widehat{(\angj{j} \psi_h^j)}\cdot[\tau_h]\,d\mathbf{s}.
        \label{eqnbi_linear)def}
    \end{equation}
    It is clear that the exact solution satisfies the weak form above, and it is linear. Then the $e^j_h$ satisfies
    \begin{equation}
        \int_{\Th}\partial_t e^j_h\tau_h\,d\bx+B_h(e^j_h,\tau_h;\angj{j})=\sigma\int_{\Th}(\overline{e_h}-e^j_h)\tau_h\,d\bx, 
    \end{equation}
    where $\overline{e_h}=\frac{1}{m(\Sang)}\sum_{j=1}^{N_{\Omega}}\omega_je^j_h$. 
    
   By setting the test function $\tau_h=\xi^j_h$
   \begin{equation}
        \label{eqn:variational_error}\int_{\Th}\partial_t\xi^j_h\xi^j_h\,d\bx+B_h(\xi^j_h,\xi^j_h;\angj{j})=\int_{\Th}\partial_t\eta^j_h\xi^j_h\,d\bx+B_h(\eta^j_h,\xi^j_h;\angj{j})+\sigma\int_{\Th}(\overline{e_h}-e^j_h)\xi^j_h\,d\bx.
    \end{equation}
    
    We first the term on the left hand-side of Equation \eqref{eqn:variational_error} which is equivalent to
    \begin{equation}
        LHS=\frac{1}{2}\frac{d}{dt}\norma{\xi^j_h}^2_{L^2(\Th)}-\int_{\Th}(\angj{j}\xi_h)\cdot\nabla_x\xi_h\,d\bx+\int_{\Eh}\widehat{(\angj{j}\xi^j_h)}[\xi^j_h]\,d\mathbf{s}
    \end{equation}
    Now  just focus on the integral terms 
\begin{align*}
&-\int_{\Th}(\angj{j}\xi^j_h)\cdot\nabla_x\xi^j_h\,d\bx+\int_{\Eh}\widehat{(\angj{j}\xi^j_h)}[\xi^j_h]\,d\mathbf{s}\\
&=-\int_{\Th}\angj{j}\cdot\nabla_x\left(\frac{(\xi^j_h)^2}{2}\right)\,d\bx+\int_{\Eh}\widehat{(\angj{j}\xi^j_h)}[\xi^j_h]\,d\mathbf{s}\\
&=\int_{\Eh}\left(-\frac{1}{2}[\angj{j}(\xi^j_h)^2]+\{\xi^j_h\angj{j}\}[\xi^j_h]+\frac{1}{2}|\angj{j}\cdot\normal|[\xi^j_h]\right)\,d\mathbf{s}\\
&=\int_{\Eh}\left(-\frac{1}{2}[\angj{j}(\xi^j_h)^2]+\frac{1}{2}[\angj{j}(\xi^j_h)^2]+\frac{1}{2}|\angj{j}\cdot\normal|[\xi^j_h]^2\right)\,d\mathbf{s}\\
&=\int_{\Eh}\frac{1}{2}|\angj{j}\cdot\normal|[\xi^j_h]^2\,d\mathbf{s}\\
\end{align*}

    then the left-hand side of the equality is given by:
    \begin{equation}
        LHS=\frac{1}{2}\frac{d}{dt}\norma{\xi^j_h}^2_{L^2(\Th)}+\int_{\Eh}\frac{1}{2}|\angj{j}\cdot\normal|[\xi^j_h]^2\,d\mathbf{s}
        \label{eqn:LHS}
    \end{equation}
    Now let us concentrate on the right hand side, since $\xi^j_h$ it is orthogonal to any piecewise polynomial of degree $k$:
    \begin{equation}
        \int_{\tau_h} \eta^j_h (\angj{j}\cdot \nabla_x\xi^j_h)\,d\bx=0,
    \end{equation}
    then the only two terms that survuve from the right hand side are
    \begin{equation}
        RHS=\int_{\Th}\partial_t\eta^j_h\xi^j_h\,d\bx+\int_{\Eh}\widehat{(\angj{j}\eta^j_h)}[\xi^j_h]\,d\mathbf{s}+\sigma\int_{\Th}(\overline{e_h}-e^j_h)\xi^j_h\,d\bx,
    \end{equation}
    For the first term in the above sum, notice that since the Time derivative commutes with $\Pi^k$, then by orthogonality and the fact that $\xi^j_h\in V_h^k$,
    \begin{equation}
        \int_{\Th}\partial_t\xi^j_h\eta^j_h\,d\bx=0, 
    \end{equation}
    for the second term, by the definition of the upwind flux \eqref{eqn:upwind_flux} and Lemma \ref{lem:approximation_properties} gives, 
    \begin{align}
    &\int_{\Eh}\widehat{(\angj{j}\eta^j_h)}[\xi^j_h]\,d\mathbf{s}\notag\\
    &=\int_{\Eh}\left(\angj{j}\{\eta^j_h\}+\frac{|\angj{j}\cdot\normal|[\eta^j_h]}{2}\right)[\xi^j_h]\,d\mathbf{s}\notag\\
    &=\int_{\Eh}\left(\{\eta^j_h\}(\angj{j}\cdot \hat{\normal})\hat{\normal}+\frac{|\angj{j}\cdot\normal|[\eta^j_h]}{2}\right)[\xi^j_h]\,d\mathbf{s}\notag\\
    &\leq \int_{\Eh}|\angj{j}\cdot\normal|\left(|\{\eta^j_h\}|+|\frac{[\eta^j_h]}{2}|\right)[\xi^j_h]\,d\mathbf{s}\notag\\
    &\leq \left(2\int_{\Eh}|\angj{j}\cdot\normal|\left(|\{\eta^j_h\}|^2+|\frac{[\eta^j_h]}{2}|^2\right)\right)^{1/2}\left(\int_{\Eh}|\angj{j}\cdot\normal|[\xi^j_h]^2\,d\mathbf{s}\right)^{1/2}\notag\\
    &\leq \left(2\int_{\Eh}\left(|\{\eta^j_h\}|^2+|\frac{[\eta^j_h]}{2}|^2\right)\right)^{1/2}\left(\int_{\Eh}|\angj{j}\cdot\normal|[\xi^j_h]^2\,d\mathbf{s}\right)^{1/2}\notag\\
    &=\left(2\int_{\Eh}\left(|{\{\eta^j_h\}^2}|\right)\right)^{1/2}\left(\int_{\Eh}|\angj{j}\cdot\normal|[\xi^j_h]^2\,d\mathbf{s}\right)^{1/2}\notag\\
    &\leq C\norma{\eta^j_h}_{L^2(\Eh)}\left(\int_{\Eh}|\angj{j}\cdot\normal|[\xi^j_h]^2\,d\mathbf{s}\right)^{1/2}\notag\\
    &\leq h^{\dgorder/2}\norma{\psij{j}}_{\dgorder,\BX}\left(\int_{\Eh}|\angj{j}\cdot\normal|[\xi^j_h]^2\,d\mathbf{s}\right)^{1/2}. \label{eqn:surface_l2_Estimate}
    \end{align}

    Finally notice that, by the orthogonality of the $L^2$ projection, $\xi^j_h\bot\eta^i_h$, then  
\begin{align}
    \int_{\Th} (\overline{e_h}-e^j_h)\eta^j_h\,d\bx&=\int_{\Th} [(\overline{\xi_h}-\xi^j_h)-(\overline{\eta^j_h}-\eta^j_h)]\xi^j_h\,d\bx\\
    &=-\int_{\Th} (\xi^j_h-\overline{\xi_h})\xi^j_h\,d\bx\\
\end{align}
we can easily see that 
\begin{align}
    \sum_{j=1}^{\Nang}\omega_j\int_{\Th}(\xi^j_h-\overline{\xi_h})\xi^j_h\,d\bx&=\int_{\Th}\left[\sum_{j=1}^{\Nang}\omega_j(\xi^j_h-\overline{\xi_h})\xi^j_h\right]\,d\bx\\
    &=\int_{\Th}\left[\sum_{j=1}^{N_{\ang}}\omega_j(\xi^j_h-\overline{\xi_h})\right]^2\,d\bx\\
    &+\int_{\Th}\left[\sum_{j=1}^{N_{\ang}}\omega_j(\xi^j_h-\overline{\xi_h})\right]\overline{\xi^j_h}\,d\bx\\
    &=\int_{\Th}\left[\sum_{j=1}^{N_{\ang}}\omega_j(\xi^j_h-\overline{\xi_h})\right]^2\,d\bx,
\end{align}

then combining our computations for the lef-hand and right-hand sides together and computing the weighted sum,  
\begin{align*}
    &\frac{1}{2}\frac{d}{dt}\sum_{j=1}^{\Nang}\omega_j\norma{\xi^j_h}^2_{L^2(\Th)}+\sum_{j=1}^{\Nang}\omega_j\int_{\Eh}\frac{1}{2}|\angj{j}\cdot\normal|[\xi^j_h]^2\,d\mathbf{s}\\
    &=\sum_{j=1}^{\Nang}\omega_j\int_{\Eh}\widehat{(\angj{j}\eta^j_h)}[\xi^j_h]\,d\mathbf{s}-\sigma\int_{\Th}\left[\sum_{j=1}^{N_{\ang}}\omega_j(\xi^j_h-\overline{\xi_h})\right]^2\,d\bx,
\end{align*}
Hence, 
\begin{align*}
    &\frac{1}{2}\frac{d}{dt}\sum_{j=1}^{\Nang}\omega_j\norma{\xi^j_h}^2_{L^2(\Th)}+\sum_{j=1}^{\Nang}\omega_j\int_{\Eh}\frac{1}{2}|\angj{j}\cdot\normal|[\xi^j_h]^2\,d\mathbf{s}+\sigma\int_{\Th}\left[\sum_{j=1}^{N_{\ang}}\omega_j(\xi^j_h-\overline{\xi_h})\right]^2\,d\bx,\\
    &=\sum_{j=1}^{\Nang}\omega_j\int_{\Eh}\widehat{(\angj{j}\eta^j_h)}[\xi^j_h]\,d\mathbf{s}.
\end{align*}
then by using \eqref{eqn:surface_l2_Estimate},
\begin{align*}
    &\frac{1}{2}\frac{d}{dt}\sum_{j=1}^{\Nang}\omega_j\norma{\xi^j_h}^2_{L^2(\Th)}+\sum_{j=1}^{\Nang}\omega_j\int_{\Eh}\frac{1}{2}|\angj{j}\cdot\normal|[\xi^j_h]^2\,d\mathbf{s}\\
    &\leq Ch^{\dgorder/2}\sum_{j=1}^{\Nang}\omega_j\norma{\psij{j}}_{\dgorder,\BX}\left(\int_{\Eh}|\angj{j}\cdot\normal|[\xi^j_h]^2\,d\mathbf{s}\right)^{1/2}\\
    &\leq Ch^{2\dgorder}\sum_{j=1}^{\Nang}\omega_j\norma{\psij{j}_h}_{\dgorder,\BX}^2+\sum_{j=1}^{\Nang}\omega_j\int_{\Eh}\frac{1}{2}|\angj{j}\cdot\normal|[\xi^j_h]^2\,d\mathbf{s}
\end{align*}
then we end up with the following Gronwall inequality
\begin{equation}
    \frac{d}{dt}\sum_{j=1}^{\Nang}\omega_j\norma{\xi^j_h}^2_{L^2(\Th)}\leq C\sum_{j=1}^{\Nang}\omega_j\norma{\xi^j_h}^2_{L^2(\Th)}+Dh^{2\dgorder}.
\end{equation}
Then using the fact that the initial error $\psij{j}_h(x,0)=\Pi^k\psi(x,\angj{j},0)$,
\begin{equation}
    \sum_{j=1}^{\Nang}\omega_j\norma{\psi(x,\angj{j},0)-\psij{j}_h(x,0)}^2_{L^2(\Th)}\leq C_1h^{2k+2}. 
\end{equation}
This give us the final result.
\end{proof}
\section{Analysis of angular error, \cref{thm:DO_error_time_dep}}
\begin{proof}
Let us define the \emph{error in angle} $e^j(x)=\psi(\bx,\mathbf{\Omega}_j)-\psi^j(\bx)$ and $\mathbf{e}(x)=(e^j(\bx))_{j=1}^{\Nang}$ Then, $e^j$ satisfies
\begin{equation}
    (e^j)_t+\angj{j}\cdot\nabla_{\bx}e^j=\sigma[\overline{\mathbf{e}}-e^j]+\eta.\label{eqn:error_in_angle}
\end{equation}
where $\overline{\mathbf{e}}=\frac{1}{m(\Sang)}\sum_{j=1}^{\Nang}\omega_j e^j$.
\begin{equation}
    \eta=\frac{\sigma}{m(\Sang)}\int_{\Sang}\psi(x,\ang)\,d\ang-\frac{\sigma}{m(\Sang)}\sum_{j=1}^{\Nang} \omega_j\psi(x,\angj{j}).
\end{equation}
if we multiply \eqref{eqn:error_in_angle} by $\omega_j e^j$ and integrate over $X$ and use the periodic boundary conditions in $\bx$, we have the following error equation, 
\begin{equation}
\frac{d}{dt}\sum_{j=1}^{\Nang}\omega_j\norma{e^j}_{L^2(\BX)}^2=\sum_{j=1}^{\Nang}\omega_j\int_{X}e^j\overline{\mathbf{e}}\,d\bx-\sum_{j=1}^{\Nang}\omega_j\norma{e^j}^2_{L^2(\BX)}+\sum_{j=1}^{\Nang}\omega_j\int_X e^j\eta\,d\bx    
\end{equation} 

Notice that since
\begin{equation}
    \sum_{j=1}^{\Nang}\omega_j e^j\overline{\mathbf{e}}-\sum_{j=1}^{\Nang}\omega_j(e^j)^2=-\sum_{j=1}^{\Nang}\omega_j(e^j-\overline{\mathbf{e}})^2,
\end{equation}

then
\begin{equation}
    \frac{1}{2}\frac{d}{dt}\sum_{j=1}^{\Nang}\omega_j\norma{e^j}^2_{L^2(\BX)}+\sigma\sum_{j=1}^{\Nang}\omega_j\norma{e^j-\overline{\mathbf{e}}}^2_{L^2(\BX)}=\sum_{j=1}^{\Nang}\omega_j\int_X e^j\eta\,d\bx,
\end{equation}
from we can easily see that 
\begin{equation}
    \frac{1}{2}\frac{d}{dt}\sum_{j=1}^{\Nang}\omega_j\norma{e^j}^2_{L^2(\BX)}\leq \sum_{j=1}^{\Nang}\omega_j\int_X e^j\eta\,d\bx.
\end{equation}
Thus that since the $\omega_j>0$, then a double application of Cauchy–Schwarz inequality, first for integrals and then for sums, gives
\begin{equation}
    \frac{1}{2}\frac{d}{dt}\sum_{j=1}^{\Nang}{\omega}_j\norma{e^j}^2_{L^2(\BX)}\leq \sum_{j=1}^{\Nang}{\omega}_j\norma{e^j}_{L^2(\BX)}\norma{\eta^j}_{L^2(\BX)}\leq \left(\sum_{j=1}^{\Nang}\omega_j\norma{e^j}_{L^2(\BX)}^2\right)^{1/2}\left(\sum_{j=1}^{\Nang}{\omega}_j\norma{\eta}_{L^2(\BX)}^2\right)^{1/2}
\end{equation}

An application of Lemma \ref{lem:sn_approx} gives
\begin{equation}
\left(\sum_{j=1}^{\Nang}{\omega}_j\norma{\eta}_{L^2(\BX)}^2\right)^{1/2}=m(\Sang)^{1/2}\norma{\eta}_{L^2(\BX)}\leq D\Nang^{-s}\left(\int_X\norma{\psi}_{H^s(\Sang)}^2(x,t)\,d\bx\right)^{1/2}.
\end{equation}
The conclusion follows using Gronwalls inequality and the fact that \mbox{$e^j(\bx,0)=0.$} 
\end{proof}

\end{document}